\documentclass[11pt, reqno]{amsart}

\usepackage{amsfonts}
\usepackage{mathrsfs}
\usepackage{graphicx}
\usepackage{verbatim}
\usepackage[abs]{overpic}
\usepackage{amsmath,amsthm,amssymb}
\usepackage[utf8]{inputenc}
\usepackage{array}
\usepackage{mathtools}
\usepackage{amssymb}
\usepackage{tikz}
\usetikzlibrary{shapes,snakes}
\usepackage{tabularray}
\usepackage{caption}
\usepackage{graphicx,calc}
\newlength\myheight
\newlength\mydepth
\settototalheight\myheight{Xygp}
\usepackage{geometry}
\allowdisplaybreaks

\newtheorem{theorem}{Theorem}[section]
\newtheorem{lemma}[theorem]{Lemma}
\newtheorem{corollary}[theorem]{Corollary}

\theoremstyle{definition}
\newtheorem{definition}[theorem]{Definition}
\newtheorem{example}[theorem]{Example}

\theoremstyle{remark}
\newtheorem{remark}[theorem]{Remark}

\numberwithin{equation}{section}

\newcommand{\embeddpdf}[3]{\raisebox{#1\mydepth}{\includegraphics[height=#2\myheight]{#3}}}

\author{Dheeraj Kulkarni}
\address{Department of Mathematics \\ Indian Institute of Science Education and Research Bhopal}
\email{dheeraj@iiserb.ac.in}

\author{Monika Yadav}
\address{Department of Mathematics \\ Indian Institute of Science Education and Research Bhopal}
\email{monika18@iiserb.ac.in}



\begin{document}
\title[Jones Polynomial and its Categorification for Legendrian Knots]{On a generalization of Jones Polynomial and its Categorification for Legendrian Knots}
\maketitle

\begin{abstract}
In this article, we explore a polynomial invariant for Legendrian knots which is a natural extension of Jones polynomial for (topological) knots. To this end, a new type of skein relation is introduced for the front projections of Legendrian knots.
Further, we give a categorification of the polynomial invariant for Legendrian knots which is a natural extension of Khovanov homology. The Thurston-Bennequin invariant of Legendrian knot appears naturally in the construction of the homology as the grade shift.

The constructions of the polynomial invariant and its categorification are natural in the sense that if we treat Legendrian knots as only knots (that is, we forget the geometry on the knots), then we recover the Jones polynomial and Khovanov homology respectively. In the end, we discuss strengths and limitations of these invariants.
\end{abstract}


\section{Introduction}
The Jones polynomial \cite{Jones1, Jones2} and its categorification due to Khovanov \cite{Khova1} have 
been at the core of clasical knot theory for quite some time. The combinatorial way to 
construct Jones polynomial via bracket polynomial due to Kauffman \cite{Kauffman1} makes the Jones polynomial 
more accessible for computation when a regular projection of a knot is given. A combinatorial construction of Khovanov Homology is available in \cite{Viro}. 
This makes the computation of Khovanov homology also accessible.

The bracket polynomial of Kauffman is constructed by way of resolving crossings inductively in a given knot diagram in two different ways until we get a bunch of planar circles (unknots). Each time a 
resolution of a crossing is performed, appropriate weights are attached in the computation 
(bracket relation) to simpler diagrams. Finally a polynomial is associated to the circle which acts as the basic element in the construction. The rules of bracket polynomial then give a 
polynomial associated to a knot diagram. To get Jones polynomial, the bracket polynomial is then multiplied by an appropriate power of the variable.

The combinatorial approach to Khovanov homology involves resolving all crossings in a given knot diagram to obtain various `states'. The knot diagram can be thought as a superposition of these states.
The states can be organized depending on the type of resolutions (horizontal and vertical) 
performed. More precisely, an appropriate bi-grading is introduced to organize the states. The 
chain groups are formed using states corresponding to the same bi-grading. Then, the differential map is constructed to encode the `interaction' between the states with different bi-gradings. Finally, the homology associated to a given knot diagram can be computed.  

To establish that the above constructions give invariants of knots upto isotopy, one shows the invariance of the polynomial and homology under Reidemeister moves as two knot diagrams 
represent isotopic knots if they are related by a sequence of Reidemeister moves and vice versa.

A smooth knot in $\mathbb{R}^3 $ equipped with standard contact structure (a particular 2-plane field) is a Legendrian knot if it is tangent to the contact structure at each point. Legendrian 
knots are classified up to Legendrian isotopy (an isotopy through a family of Legendrian knots). 
Similar to regular knot projections of (smooth) knots, Legendrian knots admit regular projections with finitely many cusps, called the front projection.
We also have a version of Reidemeister moves and Reidemeister theorem for Legendrian knots in terms of their front projections. Till this point, Legendrian knot theory proceeds parallel to classical knot theory. However, it requires more advanced techniques to distinguish Legendrian knots.
Legendrian isotopic Legendrian knots are smoothly isoptopic for an obvious reason. However, two Legendrian knots can be smoothly isotopic without having to be Legendrian isotopic. For instance,
there are infinitely many Legedrian knots (up to Legendrian isotopy) that are smoothly isotopic 
to the unknot. Thus, one may think of the classification problem of Legendrian as a refinement of the classification problem of knots.

In the light of the above discussion, one may ask if there is a refinement of Jones polynomial and Khovanov homology for Legendrian knots. The question implicitly suggest natural extensions
of the bracket polynomial and Khovanov homology in the sense that a reasonable notion of extension will reduce to
bracket polynomial and Khovanov homology of the underlying smooth knot respectively if the Legendrian knot is treated as a smooth knot only. In this article, we settle the above question by extending the Jones polynomial and Khovanov homology in a natural way to the setting of Legendrian knots.

In order to arrive at an appropriate refinement of Jones polynomial, we take simple-minded viewpoint of generalizing the bracket polynomial in the first step. We work with front
projections. We introduce resolutions of a crossing adapted to the front projections so that after resolving each crossing in a front projection, we are left with `basic elements' analogous
to the circle, that is front projections of Legendrian unknots. However, there are infinitely many Legendrian unknots (up to equivalence). For a complete classification of Legendrian unknots, see \cite{Eliashberg_Fraser}. We do not need the complete classification for our construction. It will suffice to know that there are infinitely many Legendrian unknots (up to equivalence). We would like to think about front projections as superposition of basic elements. Our bracket relations will encode the basic elements. Our first result in this direction is as follows.

\begin{theorem}\label{thm1}
Let $K$ be a Legendrian knot in $(\mathbb{R}^3,\xi).$ Let $K_F$ be its front projection, then there is a polynomial in two variables, denoted by $P_K(A,r)$, (see Section \ref{Legend_Jones_poly} for detailed definition) that is an in invariant of Legendrian knot upto Legendrian isotopy. Moreover, the substitution $r=1 $ recovers the Jones polynomial of $K$.
\end{theorem}

To generalize the Khovanov homology, we first generalize the basic states that are adapted to front projections. We introduce one more grading on Khovanov chain complex to encode Legendrian
front projections arising out of resolutions. For a Legendrian knot $K$, with a front projection $ K_F$, the extension of Khovanov chain complex is denoted by $\left( C_{i,j,k} \left( K_F\right), \partial \right) $. The 
boundary map $\partial $ of this new chain complex is adapted to front projections whose square is shown to vanish. Our treatment is inspired by the definition of Khovanov homology given in \cite{Viro}. Our next results can now be stated as follows.

\begin{theorem}\label{thm2}
The homology groups $H_{i,j,k}  \left( K_F\right) $ of the chain complex $\left( C_{i,j,k} \left( K_F\right), \partial \right) $ are independent of the choice of the front projection $K_F $ of the Legendrian knot $K$. In other words, the homology groups are invariant of the Legendrian knot $K$. We refer to the groups $\{ H_{i,j,k} (K) \} $ the Legendrian Khovanov homology of $K$.
\end{theorem}

\begin{theorem}\label{thm4}
The graded Euler characteristic of the Khovanov homology groups $H_{i,j,k}  \left( K_F\right) $ associated to a Legendrian knot $K$ gives $P_K(A,r)$. In other words, $H_{i,j,k}  \left( K_F\right) $ give a categorification of the Jones polynomial $P_K(A,r)$.

\end{theorem}

Thurston-Bennequin invariant naturally arises in the construction of the Khovanov homology groups. More precisely, it arises as the grade shift (depending only on the given Legendrian knot) in the new grading that we introduced.

\begin{theorem}\label{thm3}
For a given Legendrian knot $K$ and $i, j \in \mathbb{Z} $, $H_{i,j}(K)=H_{i,j-tb(K),j}(K)$, where $H_{i,j}(K)$ is the Khovanov homology group for the smooth knot corresponding to $K$.
\end{theorem}
As a consequence we get the following corollary.
 \begin{corollary}\label{cor1}
 Let $K$ and $K'$ be two (smoothly) isotopic knots then they have same Legendrian Khovanov homology if and only if $tb(K)=tb(K').$
 \end{corollary} 

 \noindent {\bf Acknowledgements}: The second author is supported by the CSIR grant 09/1020(0152)/2019-EMR-I, DST, Government of India.

\section{Preliminaries}
\begin{definition}
Let $M$ be a $3$-dimensional manifold. A \emph{contact
structure} on $M$ is a hyperplane field $\xi$ such that for any locally defined differential 1–form $\alpha$ with $ker\alpha= \xi $, we have the following nowhere integrability condition:
$$\alpha \wedge d\alpha  \neq 0.$$
Such an $\alpha$ is called \emph{contact form} and $(M,\xi)$ is called \emph{contact manifold}.
\end{definition}

\begin{example}
On $\mathbb{R}^3$, with coordinates $(x,y,z)$, the one form $\alpha_{0}= dz-ydx$ is a contact form and the contact structure given by this form is called the \emph{standard contact structure} on $\mathbb{R}^3$. It is denoted by $\xi_{st}$. Throughout this article, the contact form $\alpha_0 $ and contact structure $\xi_{st} $ on $\mathbb{R}^3 $ is fixed. 
\end{example}

\begin{definition}
         A \emph{Legendrian knot} in  $(\mathbb{R}^3,\xi_{st}) $ is a smooth embedding $\gamma : \mathbb{S}^1 \rightarrow M$, such that,  $\gamma '(\theta) \in \xi_{\gamma(\theta)}  \text{ for all } \theta \in \mathbb{S}^1 .$
        \end{definition}

 \begin{definition}
        Let $\gamma(t)=(x(t),y(t),z(t))$ be a parametrized Legendrian knot in $(\mathbb{R}^3,\xi_{st})$.\\
        Then a \emph{front projection} of $\gamma$ is a projection in $\mathbb{R}^2$, $\gamma_F : \mathbb{S}^1 \rightarrow \mathbb{R}^2$, given by $\gamma_F(t)=(x(t),z(t)).$\\
        \end{definition}

\begin{figure}[h!]
    \centering
    \includegraphics[scale=0.55]{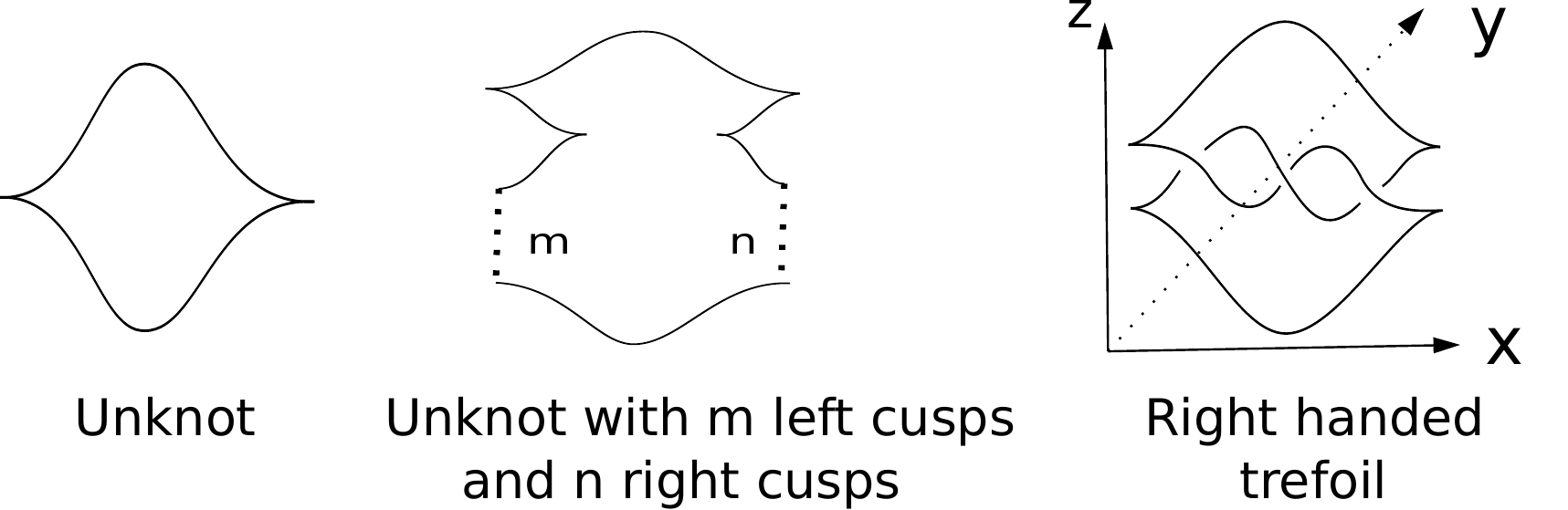}
    \caption{Examples of front projections}
    \label{fig:egoffront}
\end{figure}
Note that a parametrization of Legendrian knot $\gamma $ satisfies the equation $z'(t) = y(t)x'(t) $. Examples of \textit{generic} front projections are shown in Figure \ref{fig:egoffront}. There is no point on the
front projection where the tangent is vertical, instead we have cusps on both right and left side. This is in fact true for any generic
front projection of Legendrian knots, see \cite{Etnyre} for more details. At the cusp points the $x'(t)=z'(t)=0$ and $y=0 $. Away from the cusps, the $y$-coordinate can be
recovered by $y(t)=\frac{z'(t)}{x'(t)}$. We note that at a crossing the strand with lower slope will be above in front projection, see the third front in Figure \ref{fig:egoffront}.

\begin{definition}
A \emph{front diagram} is an immersion of disjoint union of circles in $x$-$z$ plane such that the following holds:
\begin{itemize}
    \item The immersion is an embedding at all but finitely many points. At these points, we have a cusp or a double point.
    \item There is no vertical tangency at any point.
\end{itemize}
\end{definition}
We note that every Legendrian link gives rise to a front diagram by taking a generic front projection. Further, given a front diagram, there is a Legendrian link whose front projection agrees with the front diagram. 
\begin{definition}\label{def1}
Two Legendrian knots $K,K'$ are said to be \emph{Legendrian equivalent} if there exist a smooth map $F: \mathbb{S}^1\times [0,1] \rightarrow (\mathbb{R}^3,\xi)$ such that $F_0$ is identity, $F_1(K)=K'$ and for each $t\in [0,1], F_t$ is a Legendrian embedding.
\end{definition}

\begin{remark}
 Existence of  an isotopy $F$ in Definition \ref{def1} implies existence of contact isotopy of $(\mathbb{R}^3,\xi_{st})$ taking $K$ to $K'$.
\end{remark} 
\begin{theorem}
Two Legendrian knots $K$ and $K'$ are Legendrian equivalent if and only if we can obtain $K_F$ from $K'_F$ by a sequence of finitely many local moves as shown in Figure \ref{fig:LRmoves}, where $K_F$ and $K_F'$ denote tha front projections of $K$ and $K'$ respectively.
\begin{figure}[!htbp]
    \centering
    \includegraphics[scale=0.55]{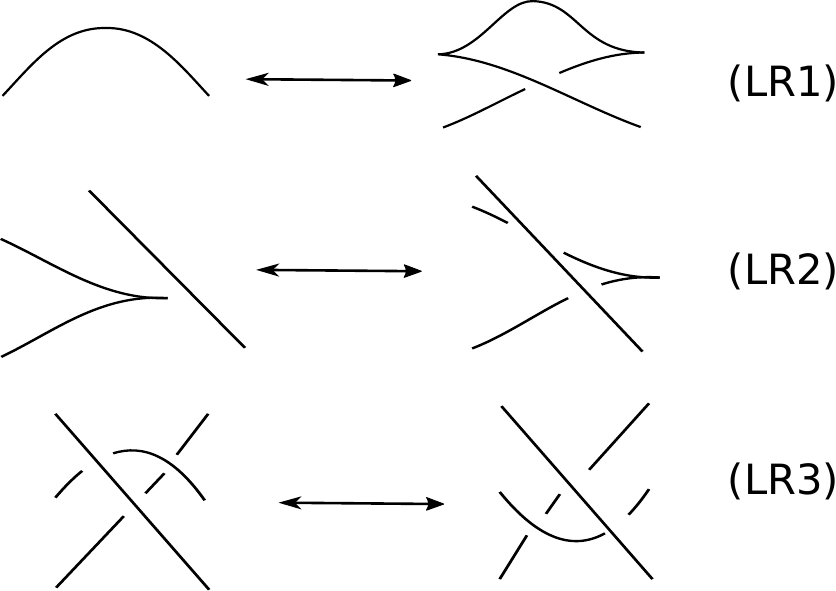}
    \caption{ LR moves for front projection}
    \label{fig:LRmoves}
\end{figure}

These moves are called Legendrian Reidemeister moves. We refer to these moves as LR1, LR2 and LR3 respectively.
\end{theorem}

\section{Construction of Jones Polynomial for Legendrian Knots}\label{Legend_Jones_poly}
 
The objective of this section is to generalize the Jones polynomial for Legendrian knots. As the Jones polynomial of a (topological) knot can be obtained from the bracket 
polynomial making appropriate variable change, we start by giving a generalization of the bracket polynomial for Legendrian knots. Since we are dealing with the front projection $K_F$, weresolve each crossing in $K_F$ in such a way that we end up with the front
projection of unknots. There are two ways of resolving a crossing, we call them $A$-resolution (horizontal resolution) and $B$-resolution (vertical resolution), see Figure \ref{fig:resolution1}. The $A$-resolution is the same as the one that we have for smooth knots but the $B$-resolution adds two more cusps to the front projection. The
reason for having this particular $B$-resolution is that we do not want the tangent to be vertical at any point of the new front projection that we get by resolving $K_F$. Note 
that we have only one type of crossing in the front projection where strand with negative slope is above the strand with positive slope.

\begin{figure}[h!]
    \centering
    \includegraphics{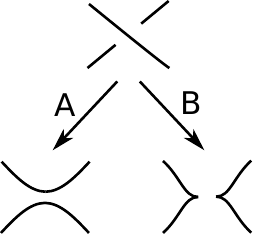}
    \caption{Resolutions of a crossing}
    \label{fig:resolution1}
\end{figure}

We now give a procedure to obtain a bracket polynomial for front projections. We list out the bracket relations below.

\begin{enumerate}
\item
  $\left\langle \embeddpdf{-4.5}{3}{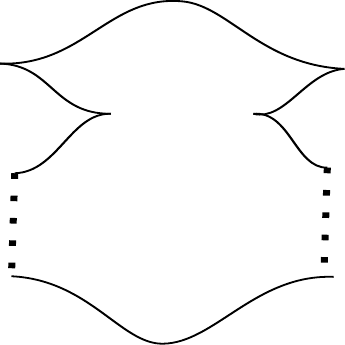} \right\rangle=-A^{2}r^{-1}-A^{-2}r,\hspace{5pt} \text{for any stabilized unknot with no crossings} $

\item
  $\left\langle K_F \sqcup \embeddpdf{-4.5}{3}{standunknot.pdf}\right\rangle=(-A^{2}r^{-1}-A^{-2}r)\langle K_F\rangle$

\item
  $\left\langle\embeddpdf{-4.5}{3}{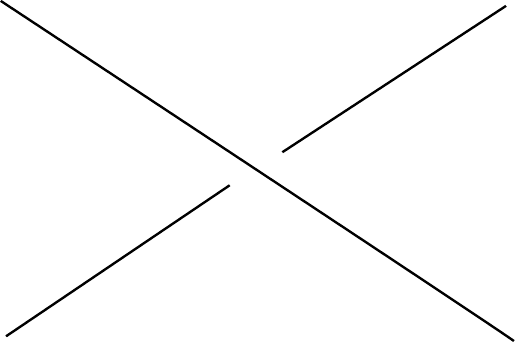}\right\rangle=
  A\left\langle \embeddpdf{-4.5}{3}{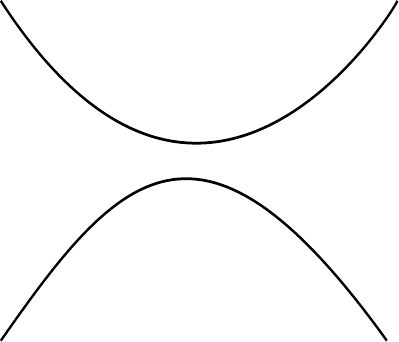}\right\rangle + A^{-1}r \left\langle \embeddpdf{-4.5}{3}{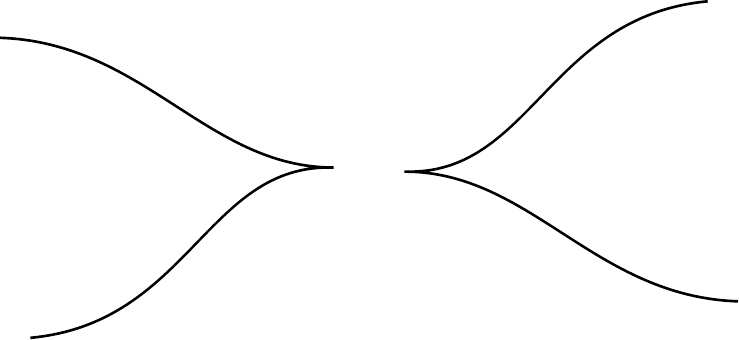}  \right\rangle$

\end{enumerate}
Given a front projection, we resolve a crossing at each step to get front diagrams with number of crossings reduced by one. By applying the above bracket relations in each step,
we arrive at a polynomial in two variables. We refer to it as the bracket polynomial. The procedure of obtaining bracket polynomial is illustrated in Example \ref{example1}, \ref{example2} and \ref{example3}.
The bracket polynomial depends on the front projection and is not an invariant of Legendrian knot upto Legendrian isotopy. However, the polynomial

\begin{eqnarray}\label{eq4}
P_K(A,r)=(-A)^{-3\omega(K_F)}r^{\frac{c\left({K_F}\right)}{2}-l(K_F)}\left\langle K_F\right\rangle
\end{eqnarray}

turns out to be an invariant of Legendrian knot upto Legendrian isotopy as claimed in Theorem \ref{thm1}.

In Equation \eqref{eq4}, the symbol $l(K_F)$ denotes the number of left handed crossings in $K_F$ and $c(K_F)$ denotes the number of cusps in $K_F$. For brevity of notation, we use the symbol $c$ instead of $c(K_F) $ whenever $ K_F$ is fixed. Note that, right handed and left
handed crossings in the front projection of a knot are independent of an orientation on the knot, however it is not the case for links with more than one components. So, now onwards we will be considering unoriented knots and oriented links.

Note that $P_K(A,r=1)$ is that bracket polynomial of the underlying smooth knot $K$ and by substituting $A=t^{-\frac{1}{4}}$ in $P_K(A,1)$
we get the Jones polynomial of $K$. We will be referring to $P_K(A,1)$ as the \textit{Jones polynomial} of the smooth knot $K$ and $P_K(A,r)$ as the \emph{Legendrian Jones polynomial} of the Legendrian knot $K$.

\begin{proof}[Proof of Theorem \ref{thm1} ]
Start with a front projection $K_F$ of a Legendrian knot $K$. We will show that $
P_K(A,r)$ does not change under the local changes created in the front projection by the Legendrian Reidemeister moves.

Under the first Legendrian Reidemeister move, the number of left handed crossings stays unchanged where the number of cusps changes by $\pm 2$ and writhe changes by $\pm 1$. Then, 
{\setlength{\abovedisplayskip}{10pt}
\begin{flalign*}
P_{\embeddpdf{-1}{1}{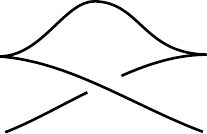}}(A,r) &=(-A)^{-3\omega(\embeddpdf{-1}{1}{Lr1for.pdf})}r^{\frac{c}{2}-l\left(\embeddpdf{-1}{1}{Lr1for.pdf}\right)}\left\langle \embeddpdf{-3}{2}{Lr1for.pdf}\right\rangle \\
&=(-A)^{-3\omega \left(\embeddpdf{-1}{1}{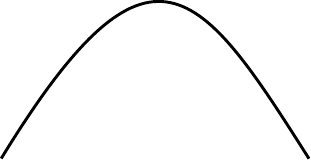} \right)-3}r^{\frac{c}{2}-l\left(\embeddpdf{-1}{1}{Lr11back.pdf}\right)+1}\left\{A\left\langle\embeddpdf{-3}{2}{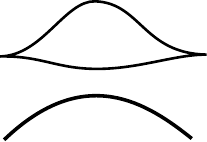}\right\rangle+A^{-1}r\left\langle\embeddpdf{-3}{2}{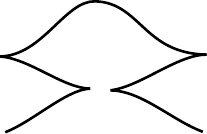}\right\rangle\right\}\\
&=(-A)^{-3\omega\left(\embeddpdf{-1}{1}{Lr11back.pdf}\right)}r^{\frac{c}{2}-l\left(\embeddpdf{-1}{1}{Lr11back.pdf}\right)}(-A)^{-3}r\left\{A(-A^{-2}r-A^{2}r^{-1})\left\langle\embeddpdf{-2}{1.5}{Lr11back.pdf}\right\rangle+A^{-1}r\left\langle\embeddpdf{-3}{2}{Lr1B.pdf}\right\rangle\right\}\\
&=(-A)^{-3\omega\left(\embeddpdf{-1}{1}{Lr11back.pdf}\right)}r^{\frac{c}{2}-l\left(\embeddpdf{-1}{1}{Lr11back.pdf}\right)}\left\langle \embeddpdf{-2}{1.5}{Lr11back.pdf}\right\rangle\\
&=P_{\embeddpdf{-1}{1}{Lr11back.pdf}}(A,r)
\end{flalign*}}
Hence, $P_K(A,r)$ is inavariant under the first Legendrian Reidemeister move.\\

Now, under the second Legendrian Reidemeister move, writhe and the number of cusps stay unchanged whereas the number of left handed crossings changes by $\pm 1$. Therefore, we have

\begin{align}\label{eq1}
\nonumber P_{\hspace{2pt}\embeddpdf{-1}{2}{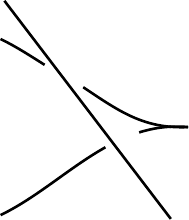}}(A,r)
&=(-A)^{-3\omega\left(\embeddpdf{-1.75}{1.5}{Lr2for1.pdf}\right)}r^{\frac{c}{2}-l\left(\embeddpdf{-1.75}{1.5}{Lr2for1.pdf}\right)}\left\langle 
\embeddpdf{-4}{3}{Lr2for1.pdf}\right\rangle \\
&=(-A)^{-3\omega\left(\embeddpdf{-1.75}{1.5}{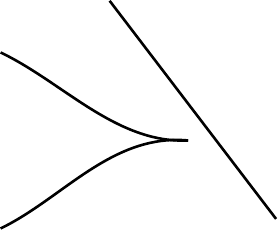}\right)}r^{\frac{c}{2}-l\left(\embeddpdf{-1.75}{1.5}{Lr2back1.pdf}\right)}r^{-1}\left\langle 
\embeddpdf{-4}{3}{Lr2for1.pdf}\right\rangle 
\end{align}

 The bracket polynomial for the front projection with two extra crossings created by  Legendrian Reidemeister  will look like
\begin{flalign*}
\left\langle \embeddpdf{-4}{3}{Lr2for1.pdf}\right\rangle &=A^2\left\langle\embeddpdf{-4}{3}{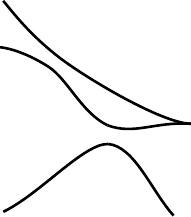}\right\rangle+A^{-1}Ar\left\langle\embeddpdf{-4}{3}{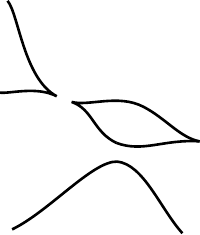}\right\rangle+A^{-1}Ar\left\langle\embeddpdf{-4}{3}{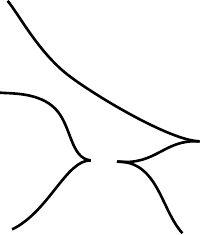}\right\rangle+A^{-2}r^2\left\langle\embeddpdf{-4}{3}{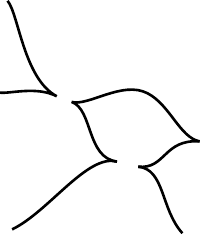}\right\rangle\\
&=A^2\left\langle\embeddpdf{-4}{3}{Lr2aa.pdf}\right\rangle+r(-A^{-2}r-A^2r^{-1})\left\langle\embeddpdf{-4}{3}{Lr2aa.pdf}\right\rangle+r\left\langle\embeddpdf{-4}{3}{Lr2ab.pdf}\right\rangle+A^{-2}r^2\left\langle\embeddpdf{-4}{3}{Lr2aa.pdf}\right\rangle\\
&=r\left\langle\embeddpdf{-4}{3}{Lr2ab.pdf}\right\rangle+\{r(-A^{-2}r-A^2r^{-1})+A^2+A^{-2}r^2\}\left\langle\embeddpdf{-4}{3}{Lr2aa.pdf}\right\rangle\\
&=r\left\langle\embeddpdf{-4}{3}{Lr2ab.pdf}\right\rangle.\\
\end{flalign*}

Thus, substituting $\left\langle \embeddpdf{-4}{3}{Lr2for1.pdf}\right\rangle=r\left\langle\embeddpdf{-4}{3}{Lr2ab.pdf}\right\rangle$ in Equation \eqref{eq1}, we get
$$P_{\hspace{2pt}\embeddpdf{-1}{2}{Lr2for1.pdf}}(A,r)=P_{\embeddpdf{-1}{2}{Lr2back1.pdf}}(A,r).$$
Therefore, we have the inavariance of $P_K(A,r)$ under second Legendrian Reidemeister move as well.

Finally, under the third Legendrian Reidemeister move, we know that the writhe, the number of cusps and the number of left handed crossings, all stay unchanged. So, we have

\begin{flalign*}
P_{\hspace{4pt}\embeddpdf{-2}{1.5}{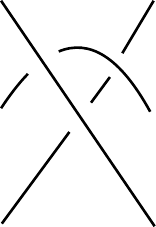}}(A,r)
&=(-A)^{-3\omega\left(\embeddpdf{-2}{1.5}{LR3back.pdf}\right)}r^{\frac{c}{2}-l\left(\embeddpdf{-2}{1.5}{LR3back.pdf}\right)}\left\langle \embeddpdf{-4}{3}{LR3back.pdf}\right\rangle \\
&=(-A)^{-3\omega\left(\embeddpdf{-2}{1.5}{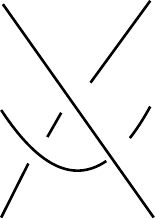}\right)}r^{\frac{c}{2}-l\left(\embeddpdf{-2}{1.5}{LR3for.pdf}\right)}\left\langle \embeddpdf{-4}{3}{LR3back.pdf}\right\rangle.
\end{flalign*}
Now, we only need to show that the bracket polynomials for the front projections before and after the third Legendrian Reidemeister move are equal. 

\begin{flalign*}
\left\langle \embeddpdf{-4}{3}{LR3for.pdf}\right\rangle =& A^3\left\langle \embeddpdf{-4}{3}{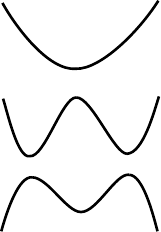}\right\rangle +Ar\left\langle \embeddpdf{-4}{3}{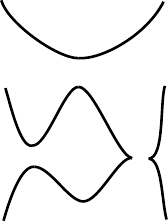}\right\rangle+Ar\left\langle \embeddpdf{-4}{3}{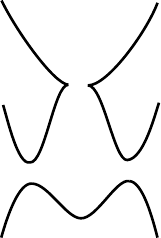}\right\rangle+Ar\left\langle \embeddpdf{-4}{3}{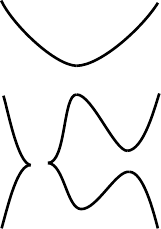}\right\rangle\\
&+A^{-1}r^2\left\langle \embeddpdf{-4}{3}{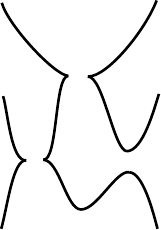}\right\rangle+A^{-1}r^2\left\langle \embeddpdf{-4}{3}{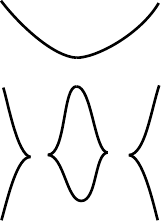}\right\rangle+A^{-1}r^2\left\langle \embeddpdf{-4}{3}{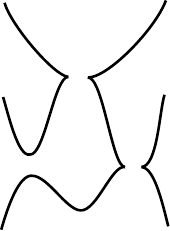}\right\rangle+A^{-3}r^3\left\langle \embeddpdf{-4}{3}{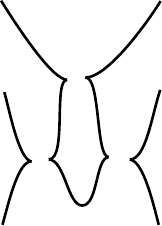}\right\rangle\\
=& A^3\left\langle \embeddpdf{-4}{3}{LR3AAA.pdf}\right\rangle +Ar\left\langle \embeddpdf{-3}{3}{LR3AAB.pdf}\right\rangle+Ar\left\langle \embeddpdf{-4}{3}{LR3ABA.pdf}\right\rangle+Ar\left\langle \embeddpdf{-4}{3}{LR3AAB.pdf}\right\rangle\\
&+A^{-1}r^2\left\langle \embeddpdf{-4}{3}{LR3BBA.pdf}\right\rangle+A^{-1}r^2(-A^{-2}r-A^2r^{-1})\left\langle \embeddpdf{-4}{3}{LR3AAB.pdf}\right\rangle\\
&+A^{-1}r^2\left\langle \embeddpdf{-4}{3}{LR3ABB.pdf}\right\rangle+A^{-3}r^3\left\langle \embeddpdf{-4}{3}{LR3AAB.pdf}\right\rangle\\
=& A^3\left\langle \embeddpdf{-4}{3}{LR3AAA.pdf}\right\rangle
+Ar\left\langle \embeddpdf{-4}{3}{LR3ABA.pdf}\right\rangle+Ar\left\langle\embeddpdf{-4}{3}{LR3AAB.pdf}\right\rangle+A^{-1}r^2\left\langle \embeddpdf{-4}{3}{LR3BBA.pdf}\right\rangle+A^{-1}r^2\left\langle \embeddpdf{-4}{3}{LR3ABB.pdf}\right\rangle
\end{flalign*}
Now, we look at the other term and expand it as follows:
\begin{flalign*}
\left\langle \embeddpdf{-4}{3}{LR3back.pdf}\right\rangle=&A^3\left\langle \embeddpdf{-4}{3}{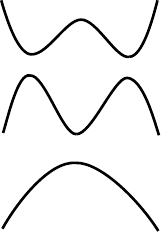}\right\rangle +Ar\left\langle \embeddpdf{-4}{3}{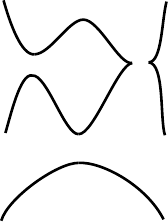}\right\rangle+Ar\left\langle \embeddpdf{-4}{3}{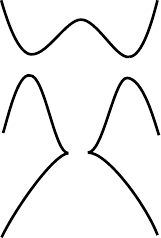}\right\rangle+Ar\left\langle \embeddpdf{-4}{3}{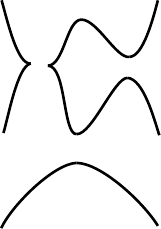}\right\rangle\\
&+A^{-1}r^2\left\langle \embeddpdf{-4}{3}{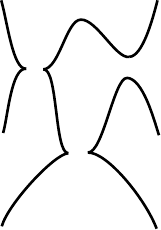}\right\rangle+A^{-1}r^2\left\langle \embeddpdf{-4}{3}{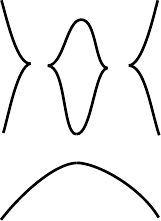}\right\rangle+A^{-1}r^2\left\langle \embeddpdf{-4}{3}{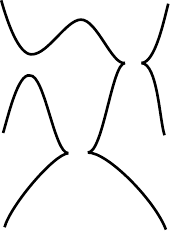}\right\rangle+A^{-3}r^3\left\langle \embeddpdf{-4}{3}{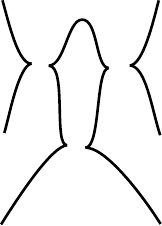}\right\rangle\\
=&A^3\left\langle \embeddpdf{-4}{3}{LR3aaa.pdf}\right\rangle +Ar\left\langle \embeddpdf{-3}{3}{LR3aab.pdf}\right\rangle+Ar\left\langle \embeddpdf{-4}{3}{LR3aba.pdf}\right\rangle+Ar\left\langle \embeddpdf{-3}{3}{LR3aab.pdf}\right\rangle\\
&+A^{-1}r^2\left\langle \embeddpdf{-4}{3}{LR3bba.pdf}\right\rangle+A^{-1}r^2(-A^{-2}r-A^2r^{-1})\left\langle \embeddpdf{-4}{3}{LR3aab.pdf}\right\rangle\\
&+A^{-1}r^2\left\langle \embeddpdf{-4}{3}{LR3abb.pdf}\right\rangle+A^{-3}r^3\left\langle \embeddpdf{-4}{3}{LR3aab.pdf}\right\rangle\\
=&A^3\left\langle \embeddpdf{-4}{3}{LR3aaa.pdf}\right\rangle+Ar\left\langle \embeddpdf{-3}{3}{LR3aab.pdf}\right\rangle+Ar\left\langle\embeddpdf{-3}{3}{LR3aba.pdf}\right\rangle+A^{-1}r^2\left\langle \embeddpdf{-4}{3}{LR3abb.pdf}\right\rangle+A^{-1}r^2\left\langle \embeddpdf{-4}{3}{LR3bba.pdf}\right\rangle\\
=&\left\langle\embeddpdf{-4}{3}{LR3for.pdf} \right\rangle
\end{flalign*}

Thus, 
\begin{flalign*}
P_{\hspace{4pt}\embeddpdf{-2}{1.5}{LR3back.pdf}}(A,r)=P_{\hspace{4pt}\embeddpdf{-2}{1.5}{LR3for.pdf}}(A,r)
\end{flalign*}
Hence, we have the invariance of the Legendrian Jones polynomial under the third Reidemeister move, which completes the proof of Theorem \ref{thm1}.
\end{proof}

Now we will discuss examples of Legendrian knots and oriented Legendrian Hopf link.
\begin{example}\label{example1}
\begin{flalign*}
\left\langle \embeddpdf{-3}{2.5}{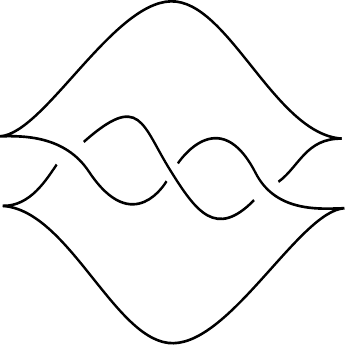}\right\rangle=&A^3\left\langle \embeddpdf{-3}{2.5}{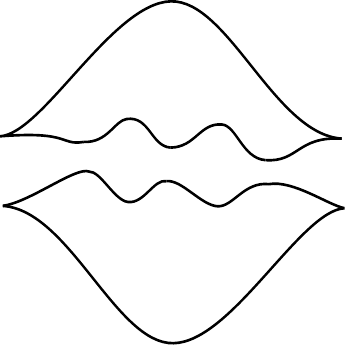}\right\rangle +Ar\left\langle \embeddpdf{-3}{2.5}{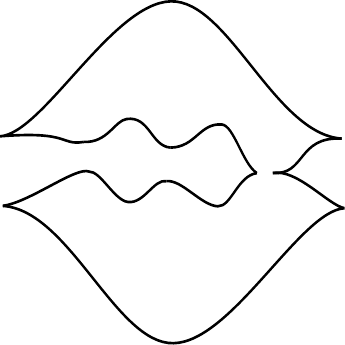}\right\rangle+Ar\left\langle \embeddpdf{-3}{2.5}{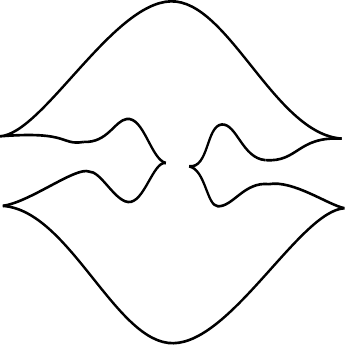}\right\rangle+Ar\left\langle \embeddpdf{-3}{2.5}{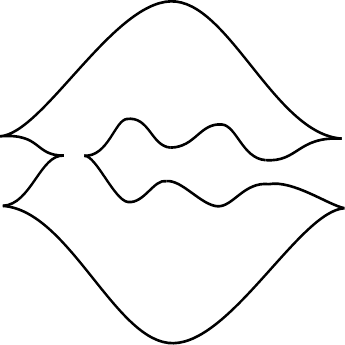}\right\rangle \\
&+A^{-1}r^2\left\langle \embeddpdf{-3}{2.5}{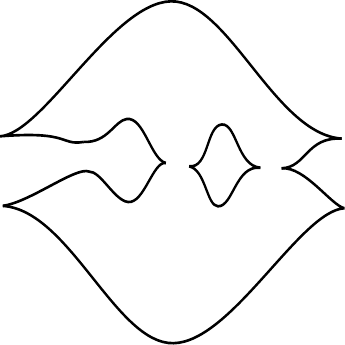}\right\rangle +A^{-1}r^2\left\langle \embeddpdf{-3}{2.5}{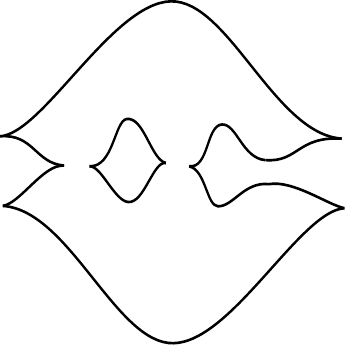}\right\rangle+A^{-1}r^2\left\langle \embeddpdf{-3}{2.5}{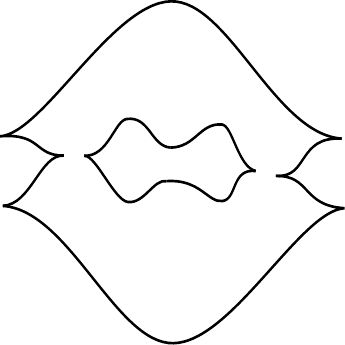}\right\rangle+A^{-3}r^3\left\langle \embeddpdf{-3}{2.5}{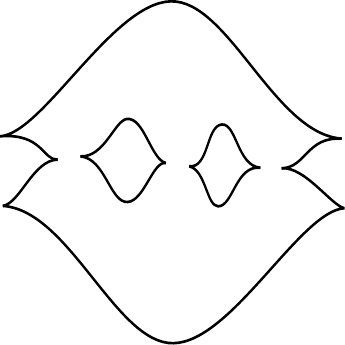}\right\rangle\\
=&A^3(-A^{-2}r-A^2r^{-1})^2+3Ar(-A^{-2}r-A^2r^{-1})+3A^{-1}r^2(-A^{-2}r-A^2r^{-1})^2\\
&+A^{-3}r^3(-A^{-2}r-A^2r^{-1})^3\\
=&(-A^{-2}r-A^2r^{-1})\{-A^5r^{-1}-A^{-3}r^3+A^{-7}r^5\}\\
\end{flalign*}
Then, 
\begin{flalign*}
P_{ \embeddpdf{-3}{2}{RHT.pdf}}(A,r)&=(-A^{-9}r^2)(-A^{-2}r-A^2r^{-1})\{-A^5r^{-1}-A^{-3}r^3+A^{-7}r^5\}\\
&=(-A^{-2}r-A^2r^{-1})\{A^{-4}r+A^{-12}r^5-A^{-16}r^7\}\\
\end{flalign*}
\end{example}

\begin{example}\label{example2}
\begin{flalign*}
\left\langle \embeddpdf{-3}{2.5}{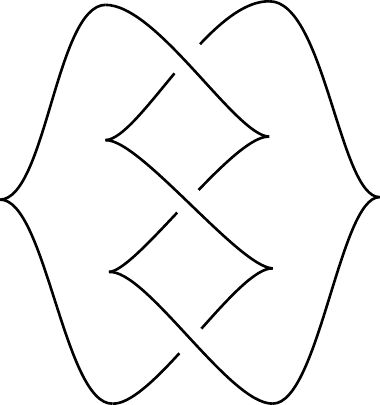}\right\rangle=&A^3\left\langle \embeddpdf{-3}{2.5}{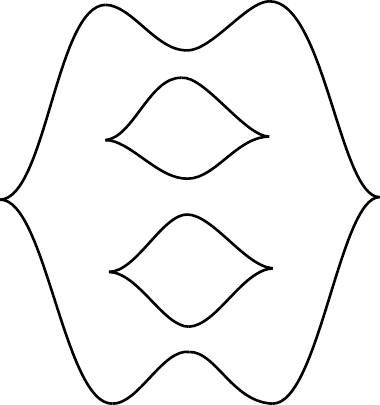}\right\rangle +Ar\left\langle \embeddpdf{-3}{2.5}{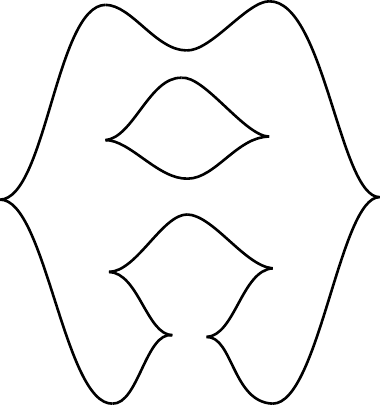}\right\rangle+Ar\left\langle \embeddpdf{-3}{2.5}{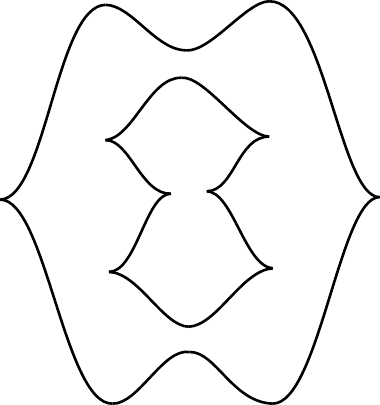}\right\rangle+Ar\left\langle \embeddpdf{-3}{2.5}{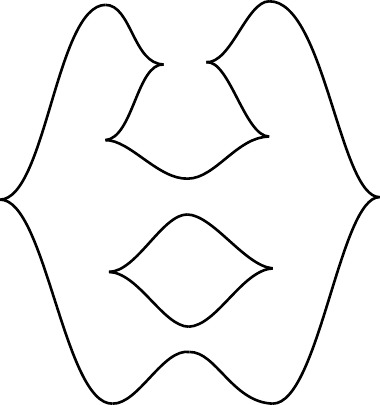}\right\rangle \\
&+A^{-1}r^2\left\langle \embeddpdf{-3}{2.5}{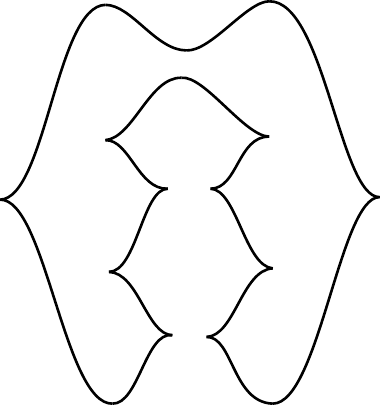}\right\rangle +A^{-1}r^2\left\langle \embeddpdf{-3}{2.5}{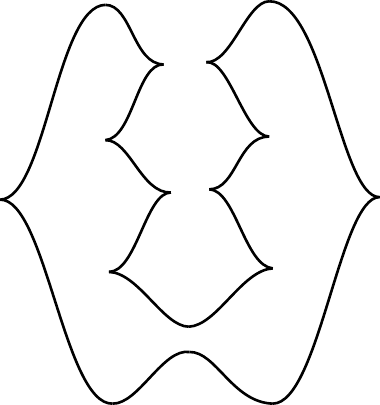}\right\rangle+A^{-1}r^2\left\langle \embeddpdf{-3}{2.5}{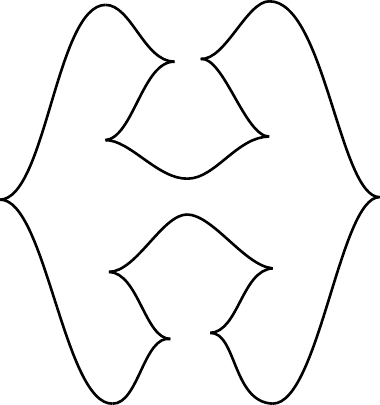}\right\rangle+A^{-3}r^3\left\langle \embeddpdf{-3}{2.5}{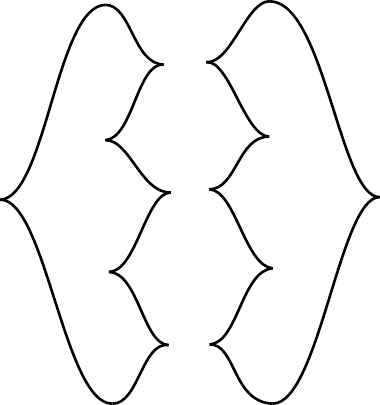}\right\rangle\\
=&A^3(-A^{-2}r-A^2r^{-1})^3+3Ar(-A^{-2}r-A^2r^{-1})^2+3A^{-1}r^2(-A^{-2}r-A^2r^{-1})\\
&+A^{-3}r^3(-A^{-2}r-A^2r^{-1})^2\\
=&(-A^{-2}r-A^2r^{-1})\{A^7r^{-2}-A^3-A^{-5}r^4\}\\
\end{flalign*}
Then, 
\begin{flalign*}
P_{ \embeddpdf{-3}{2}{lht.pdf}}(A,r)&=(-A^{-2}r-A^2r^{-1})\{-A^{16}r^{-2}+A^{12}+A^{4}r^4\}\\
\end{flalign*}
\end{example}

\begin{example}\label{example3}
\begin{flalign*}
\left\langle \embeddpdf{-3}{2.5}{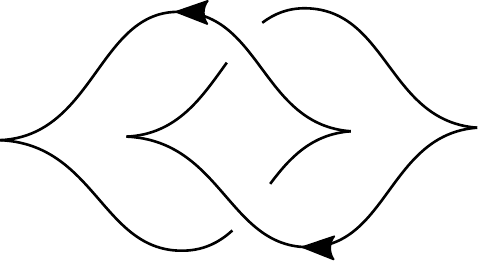}\right\rangle &=A^2\left\langle \embeddpdf{-3}{2.5}{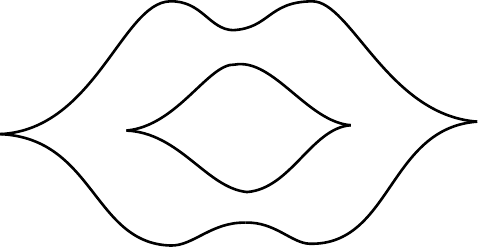}\right\rangle +r\left\langle \embeddpdf{-3}{2.5}{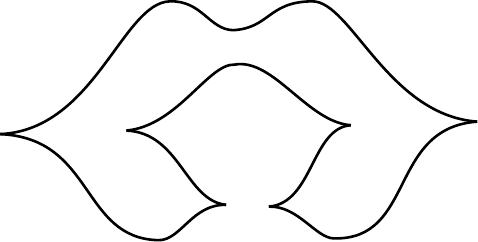}\right\rangle+r\left\langle \embeddpdf{-3}{2.5}{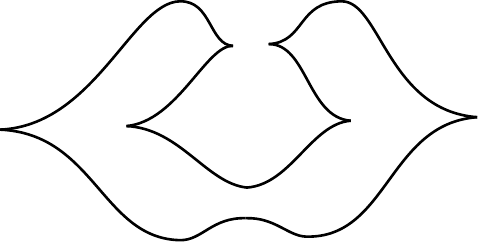}\right\rangle+A^{-2}r^2\left\langle \embeddpdf{-3}{2.5}{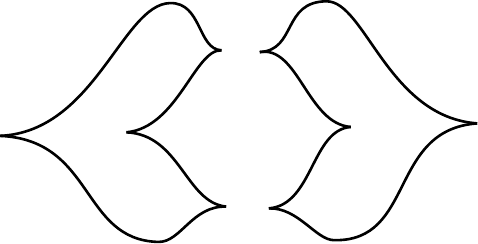}\right\rangle \\
&=A^2(-A^{-2}r-A^2r^{-1})^2+2r(-A^{-2}r-A^2r^{-1})+A^{-2}r^2(-A^{-2}r-A^2r^{-1})^2\\
&=(-A^{-2}r-A^2r^{-1})(-A^{4}r^{-1}-A^{-4}r^3)
\end{flalign*}
Then, 
\begin{flalign*}
P_{ \embeddpdf{-3}{2}{hopf.pdf}}(A,r)&=(-A^{-2}r-A^2r^{-1})(-A^{10}r^{-1}-A^2r^3)\\
\end{flalign*}

\begin{remark}
Notice that we have used a choice of orientation on the Hopf link to get the associated Jones polynomial. The orientation is used to compute the \emph{writhe} of the link which is required in the computation of Jones polynomial of the Hopf link. The same holds when we consider the Legendrian Jones polynomial of Legendrian Hopf link. However, in case of knots, the writhe does not depend on the choice of the orientation.
\end{remark}
\end{example}
 We see that in all the examples the substitution $r=1$ gives the Jones polynomials of respective smooth knots.

\section{Extension of Khovanov Homology for Legendrian knots}
After having extended the notion of Jones polynomial for the Legendrian knots, we now would like to generalize Khovanov homology to the setting of Legendrian knots in such a
way that we get a categorification of the Legendrian Jones polynomial. To achieve this, we first introduce the Legendrian state sum.
\subsection{The Legendrian State Sum}

Recall that, each crossing in the front projection can be resolved to get a new front diagrams with the number of
crossings one less than that in the previous diagram. If $n(K_F)$ denotes the number of crossings in $K_F$, then $K_F$ can be completely resolved in $2^{n(K_F)}$ ways as each crossing can
be resolved in two possible ways described in Figure \ref{fig:resolution1}.

\begin{definition}
A \emph{Legendrian state} of $K_F$ is a front diagram (which is a union of planar cusped loops) that we get after resolving all the crossings in a front projection $K_F$.
\end{definition}

The $A$-resolution carries weight $A$ and the $B$-resolution carries weight $A^{-1}r$. Then, the Legendrian bracket polynomial can be calculated by taking the sum over all the Legendrian states with their weights multiplied. For a Legendrian state $s$ of $K_F$, let $A(s)$ and $B(s)$ denote the number of $A$-resolutions and $B$-resolutions respectively in $s$. Further, $\sigma(s)$ denotes $A(s)-B(s)$ and $||s||$ denotes the number of cusped loops in $s$. Then,

\begin{align}
\nonumber \langle K_F \rangle&=\sum_{s\in \mathcal{LS}( K_F)}A^{\sigma(s)}r^{B(s)}(-A^2r^{-1}-A^{-2}r)^{||s||}. \\
P_{K}(A,r)&=(-A)^{-3\omega(K_F)}r^{\frac{c}{2}-l(K_F)}\sum_{s\in \mathcal{LS}( K_F)}A^{\sigma(s)}r^{B(s)}(-A^2r^{-1}-A^{-2}r)^{||s||}. \label{eq3}
\end{align}
 where $\mathcal{LS}( K_F)$  denotes the set of all Legendrian  states of $K_F.$

\subsection{Construction of the Legendrian Khovanov Chain Complex}
Now our objective is to realize a homology with the graded Euler characteristic equal to \eqref{eq3}. We define enhanced Legendrian states as follows.
\begin{definition}
An \emph{enhanced Legendrian state} of $K_F$ is a Legendrian state of $K_F$ with a $+$ or $-$ label for each cusped loop in it. 
\end{definition}

Note that an enhanced Legendrian state $S$ becomes a Legendrian state if we ignore the lables on it. Hence, we denote the Legendrian state corresponding to $S$ by the small letter `$s$'.

For an enhanced Legendrian state $S$, define $\tau(S)$ to be the number of $+$ labels minus the number of $-$ labels in $S$. Next we define three gradings $i(S),j(S)$ and $k(S)$ of $S$ as follows :

\begin{flalign*}
i(S)&:= \frac{\omega(K_F)-\sigma(S)}{2},\\
      j(S)&:= \frac{3\omega(K_F)-\sigma(S)+2\tau(S)}{2}, \\
 \text{ and   }  k(S)&:= \frac{c(S)-2l(K_F)+2\tau(S)}{2}.
\end{flalign*}

Let $\mathcal{ELS}(K_F)$ denote the set of the all enhanced Legendrian states arising out of the front projection $K_F $.  We rewrite Legendrian Jones polynomial \eqref{eq3}, using the substitution $-A^{-2}=q$ and the gradings $i, j $ and $k $ as follows:

\begin{align}
\nonumber P_{K}(A,r)&=(-1)^{-3\omega(k_F)}(-q^{-1})^{\frac{-3\omega(K_F)}{2}}r^{\frac{c}{2}-l(K_F)}\sum_{s\in \mathcal{LS}( K)}(-q^{-1})^{\frac{\sigma(s)}{2}}r^{B(s)}\left((qr)^{-1}+qr\right)^{||s||}\\
\nonumber &=\sum_{ S \in \mathcal{ELS}(K_F)}(-1)^{\frac{\omega(K_F)-\sigma(s)}{2}}r^{\frac{c}{2}-l(K_F)}r^{B(s)}q^{\frac{3\omega(K_F)-\sigma(s)}{2}}(qr)^{\tau(S)}\\
&=\sum_{S \in \mathcal{ELS}(K_F)}(-1)^{i(S)}q^{j(S)}r^{k(S)}. \label{eq5} 
\end{align}
In the light of the above, we construct a chain complex with coefficients from $\mathbb{Z}_2 $. 
 Let $C(K_F)$ be the abelian group generated by enhanced Legendrian states of $K_F$ over $\mathbb{Z}_2.$ 
For any $j,k,i\in \mathbb{Z},$ let $C_j(K_F)$ be the subgroup of $C(K_F)$, generated by enhanced Legendrian states $S$ with $j(S)=j,C_{k,j}\left(K_F\right)$ be the subgroup of $C_{j}(K_F)$, generated by enhanced Legendrian states $S$ with $k(S)=k$ 
 and  $C_{i,k,j}(K_F)$ be the subgroup of $C_{k,j}(K_F)$ generated by enhanced Legendrian states $S$ with $i(S)=i$. With this notation we have the following:

\begin{align*}
C(K_F)&=\bigoplus_{j\in \mathbb{Z}}C_j(K_F), \\
C_j(K_F)&=\bigoplus_{k\in \mathbb{Z}}C_{k,j}(K_F),\\
C_{k,j}(K_F)&=\bigoplus_{i\in \mathbb{Z}}C_{i,k,j}(K_F).
\end{align*}
    
   Note that $C_{i,k,j}(K_F)$ is non empty for only finitely many $i,k,j\in \mathbb{Z}$ as there are only finitely many crossings in a front projection $K_F.$

\subsection{The boundary map}
 If we ignore the geometric aspect of a Legendrian knot, it reduces to a smooth knot. Therefore, we define the boundary map $\partial: C(K_F)\rightarrow C(K_F)$ such that forgetting the third grading $k$, we get back the boundary map for the Khovanov complex of the underlying smooth knot. In order to define a boundary map, we introduce the notion of incidence of enhanced Legendrian states. To define incidence, we observe that
 groups $C_{i,k,j}(K_F)$ are same as $C_{i,j}(K_F)$ as follows.

 \begin{lemma}
Let $S$ and $T$ be two Legendrian enhanced states of $K_F$. Then $j(S)=j(T)$ if and only if $k(S)=k(T).$
\end{lemma}
\begin{proof}

Let $j(S)=j(T)$. Then, we have the following:
\begin{flalign*}
&j(S)=j(T)\\
\Longleftrightarrow &\frac{3\omega(K_F)-\sigma(S)+2\tau(S)}{2}=\frac{3\omega(K_F)-\sigma(T)+2\tau(T)}{2}\\
\Longleftrightarrow &-\sigma(S)+2\tau(S)=-\sigma(T)+2\tau(T)\\
\Longleftrightarrow& -\left((n(K_F)-B(S))-B(S)\right)+2\tau(S)= -\left((n(K_F)-B(T)\right)-B(T))+2\tau(T)\\
\Longleftrightarrow &B(S)+\tau(S)=B(T)+\tau(T)\\
\Longleftrightarrow &-l(K_F)+\left(\frac{c(K_F)}{2}+B(S)\right)+\tau(S)=-l(K_F)+\left(\frac{c(K_F)}{2}+B(T)\right)+\tau(T)\\
\Longleftrightarrow &-l(K_F)+\frac{c(S)}{2}+\tau(S)=-l(K_F)+\frac{c(T)}{2}+\tau(T)\\
\Longleftrightarrow &k(S)=k(T). &&\qedhere
\end{flalign*}
\end{proof}

\begin{definition}
A Legendrian state $t$ is \emph{incident} with a Legendrian state $s$ if $s$ and $t$ differ in resolutions at exactly one crossing and at this crossing $s$ has $A$-resolution and $t$ has $B$-resolution. We will call such a crossing the \emph{site of incidence}. We denote it by $x(s:t)$.
\end{definition} 

\begin{definition}
An enhanced Legendrian state $T$ is said to be \emph{incident} with another enhanced Legendrian state $T$ if the following conditions hold:  
    \begin{enumerate}
        \item The Legendrian states $t$ corresponding to $T$ is incident with the Legendrian state $s$ corresponding to $S$.
        \item $j(S)=j(T)$ and $k(S)=k(T)$.
        \item The labels of common cusped loops in $S$ and $T$ are same.
    \end{enumerate}
\end{definition}

    \begin{lemma}
   Let $T$ be an enhanced state incident with the enhanced state $S$, then $\sigma(S)-2=\sigma(T)$ and $ i(S)+1=i(T)$.
   \end{lemma}
   \begin{proof}
   Let $T$ be any enhanced Legendrian state incident with $S$. As $T$ has excatly one more $B$-resolution then $S$, we get $\sigma(S)-2=\sigma(T)$ and 
   \begin{align*}
   i(T)&=\frac{\omega(K_F)-\sigma(T)}{2}\\
   &=\frac{\omega(K_F)-\sigma(S)+2}{2}\\
   &=\frac{\omega(K_F)-\sigma(S)}{2}+1\\
   &=i(S)+1. &&\qedhere
   \end{align*}
   \end{proof}
   
   \begin{lemma}
   Let $T$ be incident with $S$, then $\tau(S)=\tau(T)+1$.
   \end{lemma}
   \begin{proof}
   Let $T$ be any enhanced Legendrian state which is incident with $S$. Then ,
    \begin{flalign*}
   &\Longrightarrow i(S)+1=i(T) \text{   and   } j(S)=j(T)\\
  &\Longrightarrow \sigma(S)-2=\sigma(T) \text{ and } \frac{-\sigma(S)}{2}+\tau(S)=\frac{-\sigma(T)}{2}+\tau(T)\\
  &\Longrightarrow\tau(S)=\tau(T)+1  &&\qedhere
   \end{flalign*}
    \end{proof}
     The above lemma shows that Figure \ref{fig:incident} gives the list of all possible incident states $T$ with $S.$
   \begin{figure}[!tbp]
  \centering
  \begin{minipage}[t]{0.3\textwidth} 
    \includegraphics[width=\textwidth]{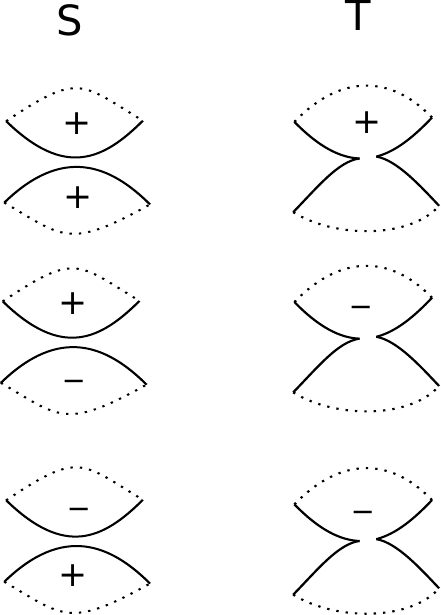}\\ 
  \end{minipage}
  \hfill
  \begin{minipage}[t]{0.4\textwidth}
    \includegraphics[width=\textwidth]{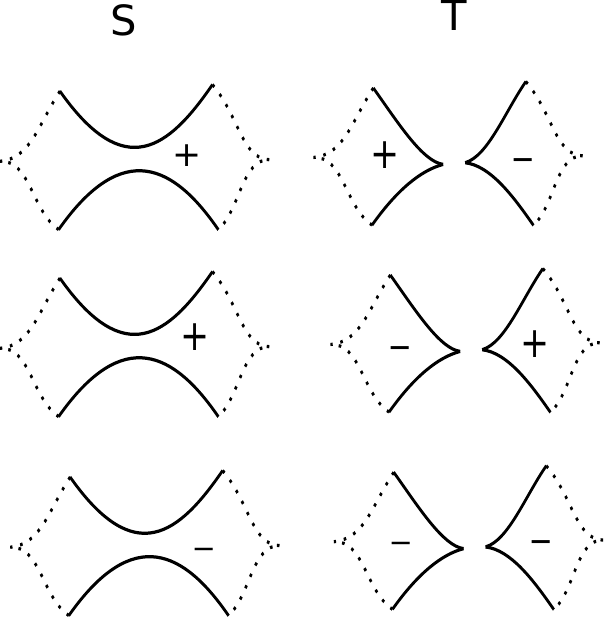}\\
  \end{minipage}
 \caption{List of all the incident states}
 \label{fig:incident}
  \end{figure}
  
    \begin{definition}
    Define, the \emph{incident number} $(S:T)$ of $T$ with $S$ to be $1$ if $T$ is incident with $S$ and $0$ otherwise.
    \end{definition}
    
    Define the boundary map $\partial : C(K_F)\rightarrow C(K_F)$ as follows: $$\partial(S)=\sum_{T \in \mathcal{ELS}(K_F)}(S:T)T.$$ The boundary of $S$ is the sum of all enhanced Legendrian states which are incident with $S$. Each term $T$ in the boundary of  $S$ is obtained by changing an $A$-resolution in $S$ to $B$-resolution without changing the labels of  unaffected loops in $S$ and assigning labels to new loop(s) such that $\tau(T)=\tau(S)-1$.
    
\begin{lemma}\label{b2}
With the above notation, we have $\partial^2=0$.
 \end{lemma}
 
 \begin{proof}
We adapt the proof for the boundary map of the Khovanov chain complex (see Theorem 5.3.A in \cite{Viro}) to our case. Let $S$ be any enhanced Legendrian state of $K_F$. Then $\partial^2(S)=\sum (S:T)(T:U)U$. To prove $\partial^2(S)=0$, we will show that each
non-zero term $(S:T)(T:U)U$ occurs even number of times. Now, for  any term $(S:T)(T:U)U$ in $\partial^2(S)$, $(S:T)(T:U)\neq 0$ if and only if T is incident with $S$ and $U$ is
incident with $T$. Then $x(S:T)\neq x(T:U)$ as $T$ has $B$-resolution at $x(S:T)$ and $A$-resolution at $x(T:U)$. As said before we will be using $s,t$ and $u$ to denote the
Legendrian states corresponding to $S,T$ and $U$ respectively. Consider the unique Legendrian state $t'$ which coincides with $t$ in terms of resolutions except at $x(S:T)$ and $x(T:U)$. Here the Legendrian state $t'$ is incident with $s$ and $u$ is incident
with $t'$. Now $x(S:T)$ and $x(T:U)$ can be part of one, two or three loops in $s$ and $u$, all these cases are described in Figure \ref{fig:case1},\ref{fig:case2},\ref{fig:case3},\ref{fig:case4} and \ref{fig:case5}. In these figures, the middle two Legendrian states are $t$ and $t'$ and we have used $x$ for $x(S:T)$ and $y$ for $x(T:U)$.

\begin{figure}[h]
    \centering
    \includegraphics[scale=0.7]{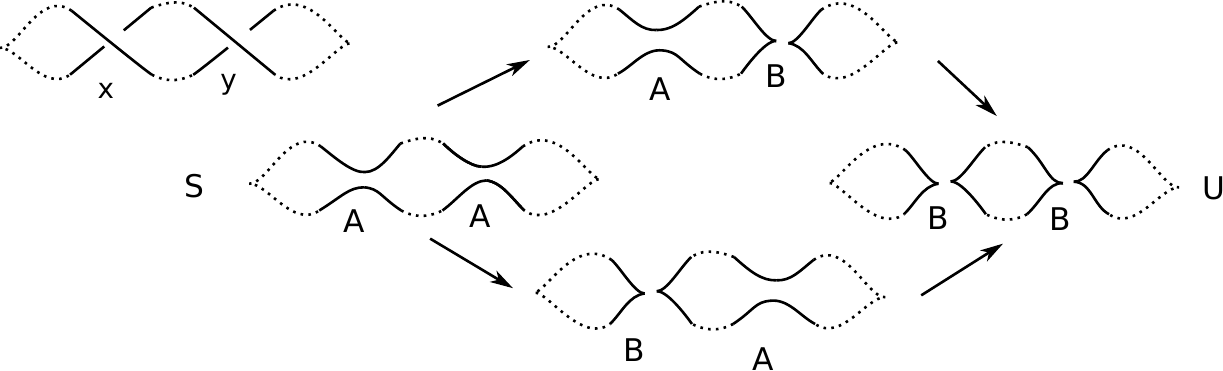}
    \caption{Case 1. $x$ and $y$ are part of one loop in $S$ and three loops in $U$}
    \label{fig:case1}
\end{figure}
\begin{figure}[h]
    \centering
    \includegraphics[scale=0.7]{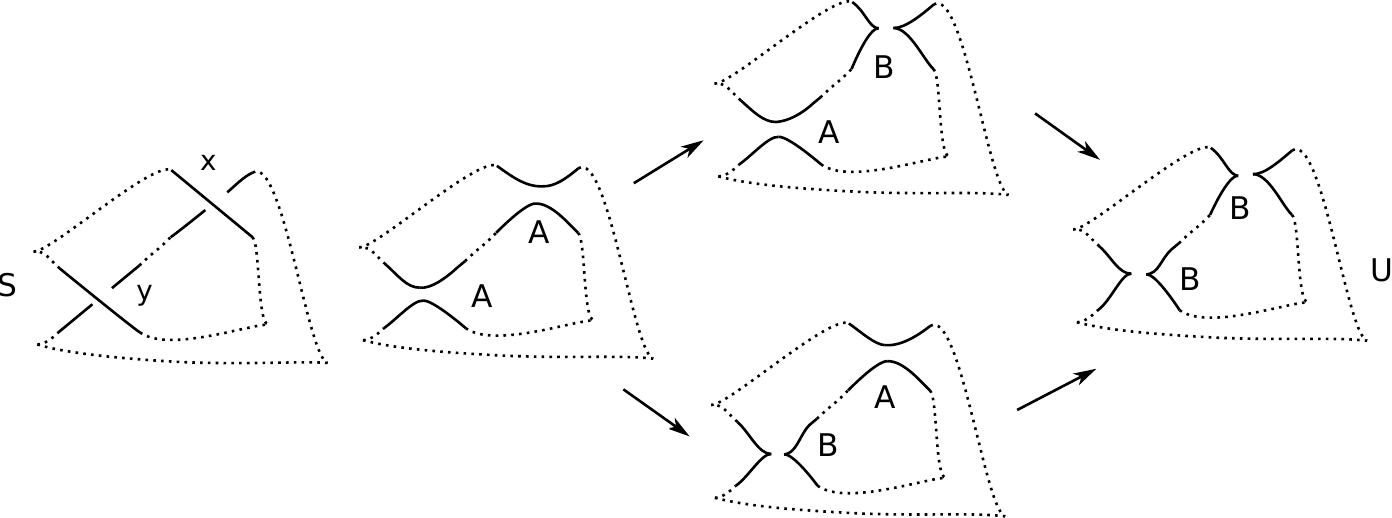}
    \caption{Case 2. $x$ and $y$ are part of one loop in both $S$ and $U$}
    \label{fig:case2}
\end{figure}
\begin{figure}[h]
    \centering
    \includegraphics[scale=0.7]{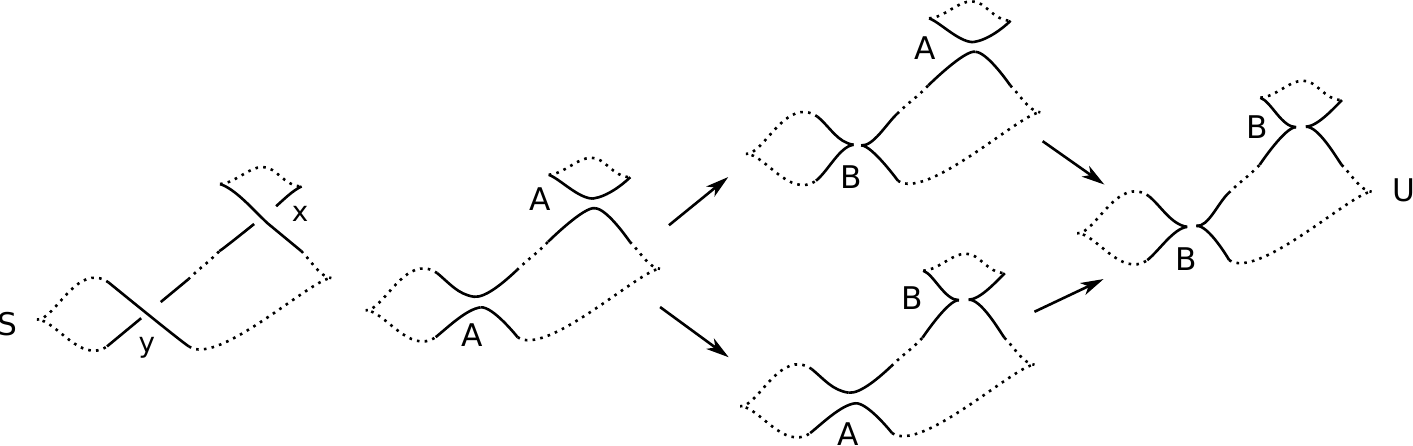}
    \caption{Case 3. $x$ and $y$ are part of two loops in $S$ and two loops in $U$}
    \label{fig:case3}
\end{figure}
\begin{figure}[h]
    \centering
    \includegraphics[scale=0.7]{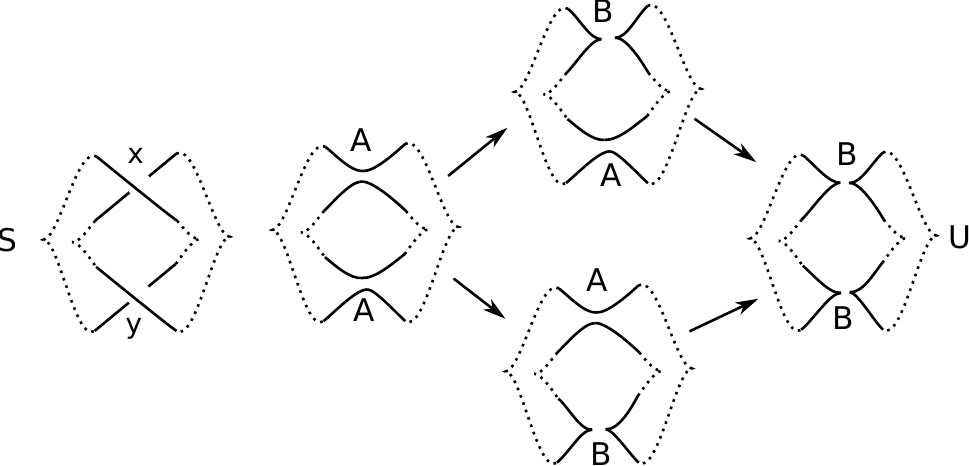}
    \caption{Case 4. $x$ and $y$ are part of two loops in $S$ and two loops in $U$}
    \label{fig:case4}
\end{figure}
\begin{figure}[h]
    \centering
    \includegraphics[scale=0.8]{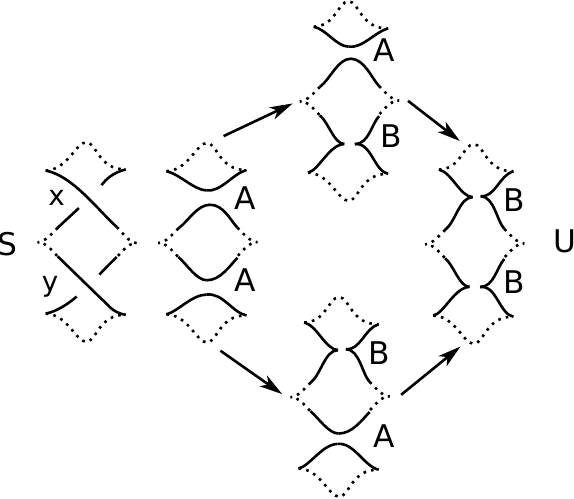}
    \caption{Case 5. $x$ and $y$ are part of three loops in $S$ and one loop in $U$}
    \label{fig:case5}
\end{figure}
 This shows that we have even number of Legendrian states which are incident with $s$ and $u$ is incident with them. Now, we show the
 same holds for enhanced Legendrian states $S$ and $U$. It can be shown by  considering all possible labels on the Legendrian state $s$ one by one.
 Below we have explained Case 1, where $S$ appears with label $+$. We can see that there are even number of enhanced Legendrian states which are incident with $S$, and $U$ is incident with them.
\begin{flalign*}
\embeddpdf{-2}{1.5}{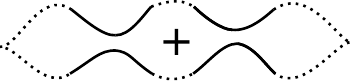}\longrightarrow& \embeddpdf{-2}{1.5}{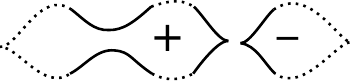}+\embeddpdf{-2}{1.5}{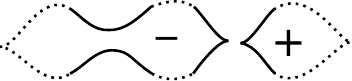}+\\
&\embeddpdf{-2}{1.5}{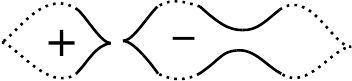}+
\embeddpdf{-2}{1.5}{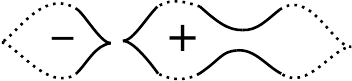}\\
&\longrightarrow 2\embeddpdf{-2}{1.5}{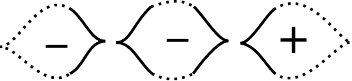}+2\embeddpdf{-2}{1.5}{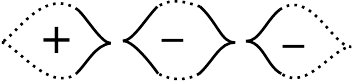}+2\embeddpdf{-2}{1.5}{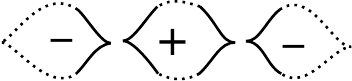}
\end{flalign*}

Calculations for some of the cases are explained below, the rest can be verified analogously.

\begin{flalign*}
\embeddpdf{-2}{1.5}{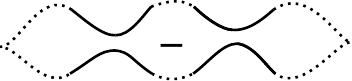}\longrightarrow \embeddpdf{-2}{1.5}{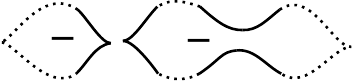}+\embeddpdf{-2}{1.5}{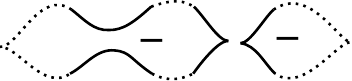}\longrightarrow 2\embeddpdf{-2}{1.5}{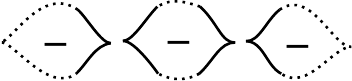}
\end{flalign*}

\begin{flalign*}
\embeddpdf{-3.5}{3}{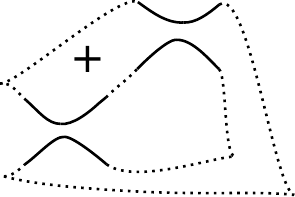}\longrightarrow \embeddpdf{-3.5}{3}{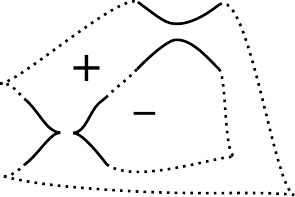}+\embeddpdf{-3.5}{3}{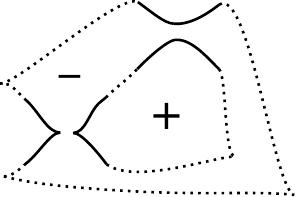}+\embeddpdf{-3.5}{3}{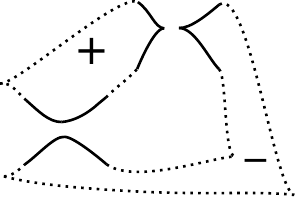}+\embeddpdf{-3.5}{3}{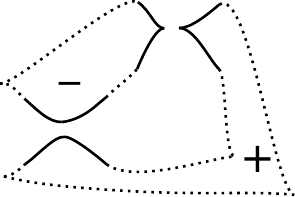}\longrightarrow 4\embeddpdf{-3.5}{3}{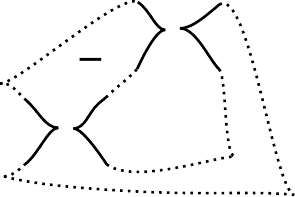}
\end{flalign*}

\begin{flalign*}
\embeddpdf{-3}{3.5}{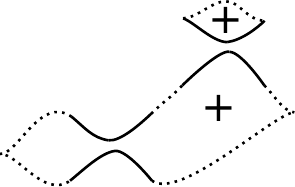}\longrightarrow \embeddpdf{-3}{3.5}{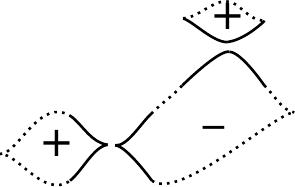}+\embeddpdf{-3}{3.5}{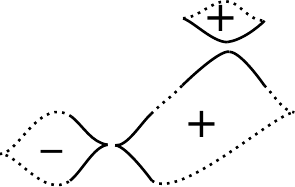}+\embeddpdf{-3}{3.5}{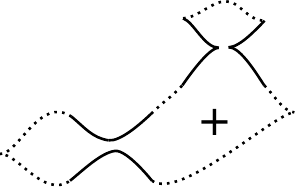}\longrightarrow 2\embeddpdf{-3}{3.5}{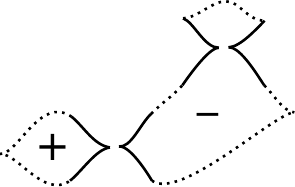}+2\embeddpdf{-3}{3.5}{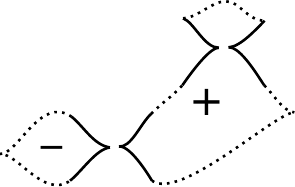}
\end{flalign*}

\begin{flalign*}
\embeddpdf{-3}{3.5}{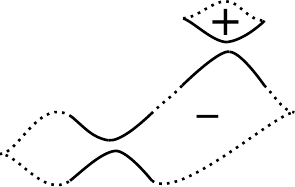}\longrightarrow \embeddpdf{-3}{3.5}{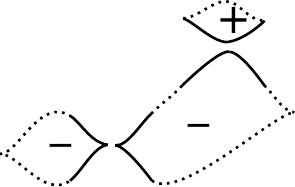}+\embeddpdf{-3}{3.5}{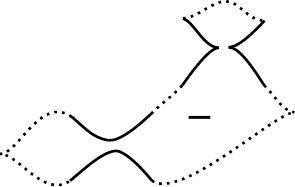}\longrightarrow 2\embeddpdf{-3}{3.5}{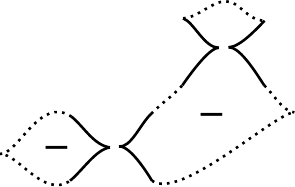}
\end{flalign*}

\begin{flalign*}
\embeddpdf{-5}{3}{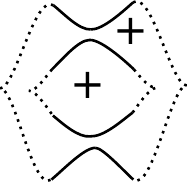}\longrightarrow \embeddpdf{-5}{3}{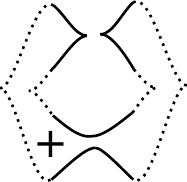}+\embeddpdf{-5}{3}{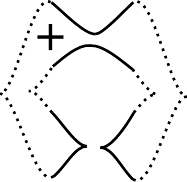}\longrightarrow 2\embeddpdf{-5}{3}{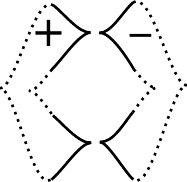}+2\embeddpdf{-5}{3}{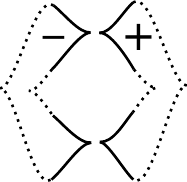}
\end{flalign*}

\begin{flalign*}
\embeddpdf{-5}{3}{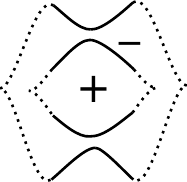}\longrightarrow \embeddpdf{-5}{3}{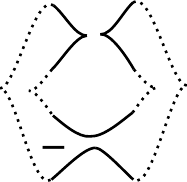}+\embeddpdf{-5}{3}{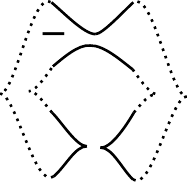}\longrightarrow 2\embeddpdf{-5}{3}{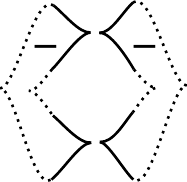}
\end{flalign*}

\begin{flalign*}
\embeddpdf{-8}{4.5}{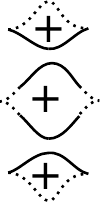}\longrightarrow \embeddpdf{-8}{4.5}{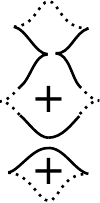}+\embeddpdf{-8}{4.5}{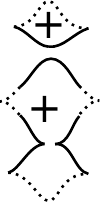}\longrightarrow 2\embeddpdf{-8}{4.5}{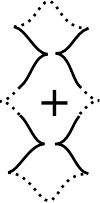}
\end{flalign*}

\begin{flalign*}
\embeddpdf{-8}{4.5}{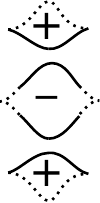}\longrightarrow \embeddpdf{-8}{4.5}{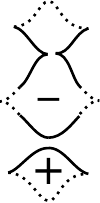}+\embeddpdf{-8}{4.5}{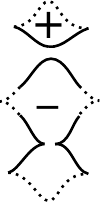}\longrightarrow 2\embeddpdf{-8}{4.5}{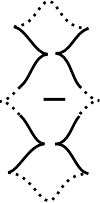}
\end{flalign*}
 \end{proof}

We refer to $(C(K_F), \partial)$, where the $\partial $ is the boundary map shown above, as the \emph{Legendrian Khovanov chain complex} of $K_F$.
\begin{theorem}\label{thm5}
Given a front projection $K_F$ of a Legendrian knot $K$ in $(\mathbb{R}^3,\xi_{st})$, there exist a graded Homology $H_{i,k,j}(K_F)$ associated to it.
\end{theorem}

\begin{proof}
It is easy to see that $\partial$ takes $C_{i,k,j}(K_F)$ to $C_{i+1,k,j}(K_F)$. Then we can define, $\partial_{i,k,j}:=\partial:C_{i,k,j}(K_F)\rightarrow C_{i+1,k,j}(K_F)$ and $\partial_{i,k,j}\circ\partial_{i-1,k,j}=0$. Thus we have $H_{i,k,j}(K_F):=\frac{ker(\partial_{i,k,j)}}{Im(\partial_{i-1,k,j})}$ a graded homology associated to $K_F$. 
\end{proof}

From now onwards, we will be referring to the homology of Legendrian Khovanov chain complex $C_{i,k,j}(K_F)$ as the \emph{Legendrian Khovanov homology} associated to the front projection $K_F $.

\section{Invariance of the Legendrian Khovanov homology under LR moves}
After proving the existence of a graded homology $\{H_{i,k,j}(K_F)\}$, our next goal is to prove that this homology is independent of the front project $K_F$. To this end, we show that application of any Legendrian Reidemeister move on $K_F$ does not change the Legendrian Khovanov homology. 
\begin{proof}[Proof of Theorem \ref{thm2}]
 First, we consider the LR1 move.
The Legendrian Khovanov complex $C \left(\embeddpdf{-2}{2}{Lr1for.pdf}\right)$ splits as follows
$$C \left(\embeddpdf{-2}{2}{Lr1for.pdf}\right)=C\left(\embeddpdf{-2}{2}{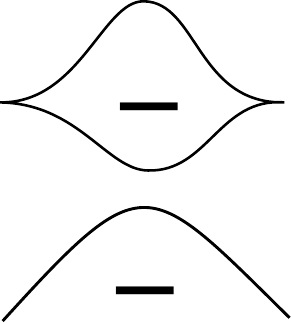},\embeddpdf{-2}{2}{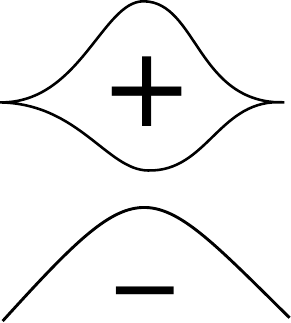}-\embeddpdf{-2}{2}{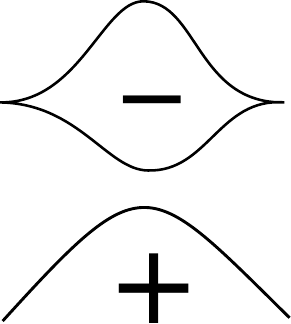}\right)\bigoplus C\left(\embeddpdf{-2}{2}{Lr1-+a.pdf},\embeddpdf{-2}{2}{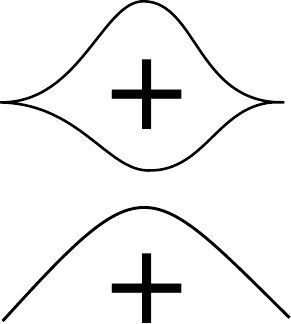},\embeddpdf{-2}{2}{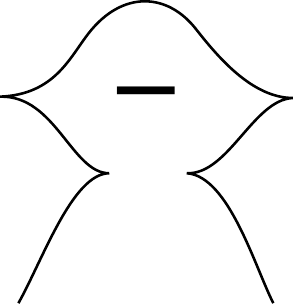},\embeddpdf{-2}{2}{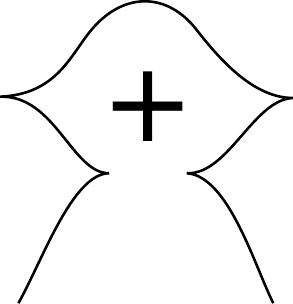}\right)$$ 
The Legendrian Khovanov complex $C \left(\embeddpdf{-2}{2}{Lr1for.pdf}\right)$ deformation retracts onto $C\left(\embeddpdf{-2}{2}{Lr1--a.pdf},\embeddpdf{-2}{2}{Lr1+-a.pdf}-\embeddpdf{-2}{2}{Lr1-+a.pdf}\right)$, the retraction map $\rho_1 :C \left(\embeddpdf{-2}{2}{Lr1for.pdf}\right) \rightarrow C\left(\embeddpdf{-2}{2}{Lr1--a.pdf},\embeddpdf{-2}{2}{Lr1+-a.pdf}-\embeddpdf{-2}{2}{Lr1-+a.pdf}\right)$ is given by :

\begin{flalign*}
&\rho_1\left(\embeddpdf{-2}{2}{Lr1--a.pdf}\right)=\embeddpdf{-2}{2}{Lr1--a.pdf}, \quad \rho_1\left(\embeddpdf{-2}{2}{Lr1-+a.pdf}\right)=\embeddpdf{-2}{2}{Lr1-+a.pdf}-\embeddpdf{-2}{2}{Lr1+-a.pdf},\\
&\rho_1\left(\embeddpdf{-2}{2}{Lr1++a.pdf}\right),\; \rho_1\left(\embeddpdf{-2}{2}{Lr1+-a.pdf}\right), \;
\rho_1\left(\embeddpdf{-2}{2}{Lr1-b.pdf}\right), \;
\rho_1\left(\embeddpdf{-2}{2}{Lr1+b.pdf}\right)=0.
\end{flalign*}

The map $i\circ \rho_1 :C \left(\embeddpdf{-2}{2}{Lr1for.pdf}\right) \rightarrow C \left(\embeddpdf{-2}{2}{Lr1for.pdf}\right) $ is homotopic to the map $id:C \left(\embeddpdf{-2}{2}{Lr1for.pdf}\right)\rightarrow C \left(\embeddpdf{-2}{2}{Lr1for.pdf}\right)$. This homotopy map $h_1:C \left(\embeddpdf{-2}{2}{Lr1for.pdf}\right)\rightarrow C \left(\embeddpdf{-2}{2}{Lr1for.pdf}\right)$, is defined on generators as: 

\begin{flalign*}
&h_1\left(\embeddpdf{-2}{2}{Lr1-b.pdf}\right)=\embeddpdf{-2}{2}{Lr1+-a.pdf}, \quad
h_1\left(\embeddpdf{-2}{2}{Lr1+b.pdf}\right)=-\embeddpdf{-2}{2}{Lr1++a.pdf},\\
&h_1\left(\embeddpdf{-2}{2}{Lr1++a.pdf}\right), \;
h_1\left(\embeddpdf{-2}{2}{Lr1+-a.pdf}\right), \;
h_1\left(\embeddpdf{-2}{2}{Lr1-+a.pdf}\right), \;
h_1\left(\embeddpdf{-2}{2}{Lr1--a.pdf}\right)=0.
\end{flalign*}  

It can be checked that on each generator $\partial h_1+h_1\partial=id-i\circ \rho_1$, that is, $h_1$ is a 
homotopy bewteen $i\circ \rho_1$  and $id$. The map $\Phi_{1} : C\left(\embeddpdf{-2}{2}{Lr1--a.pdf},\embeddpdf{-2}{2}{Lr1+-a.pdf}-\embeddpdf{-2}{2}{Lr1-+a.pdf}\right) \rightarrow C \left(\embeddpdf{-2}{2}{Lr11back.pdf}\right) $, defined on generators as
\begin{align*}
\Phi_1\left(\embeddpdf{-2}{2}{Lr1--a.pdf}\right)&= \embeddpdf{-1}{1}{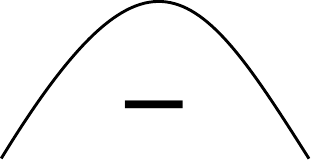},\qquad
\Phi_1\left(\embeddpdf{-2}{2}{Lr1+-a.pdf}-\embeddpdf{-2}{2}{Lr1-+a.pdf}\right)= \embeddpdf{-1}{1}{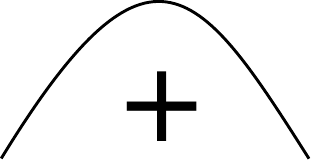},
\end{align*}  
 is an isomorphism and it satisfies $\Phi_1 \circ\partial=\partial \circ \Phi_1$. Therefore, the Legendrian Khovanov homology is invariant under LR1.

Legendrian Khovanov complex  $C\left(\embeddpdf{-2}{2}{Lr2for1.pdf}\right)$ splits into direct sum of its subcomplexes $C_r\left(\embeddpdf{-2}{2}{Lr2for1.pdf}\right)$ and $C_{cr}\left(\embeddpdf{-2}{2}{Lr2for1.pdf}\right)$, where $C_{cr}\left(\embeddpdf{-2}{2}{Lr2for1.pdf}\right)$ is generated by enahanced Legendrian states shown as below:
\begin{flalign*}
 &\embeddpdf{-4}{3}{Lr2aa.pdf}, \quad \embeddpdf{-4}{3}{Lr2bb.pdf}, \quad
 \embeddpdf{-4}{3}{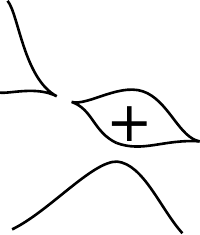}, \quad \embeddpdf{-4}{3}{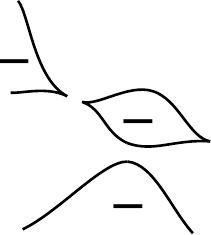}, \quad
 \embeddpdf{-4}{3}{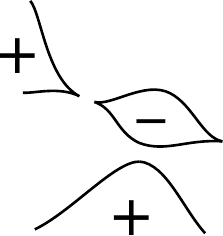}+\embeddpdf{-4}{3}{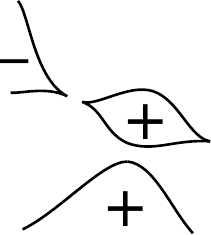}+
 \embeddpdf{-4}{3}{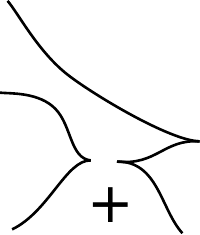}+\embeddpdf{-4}{3}{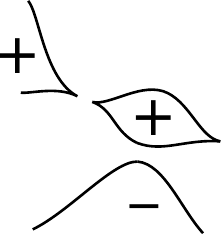},\quad
 \embeddpdf{-4}{3}{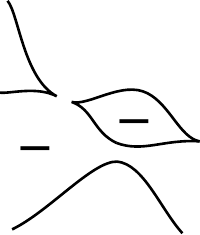}+\embeddpdf{-4}{3}{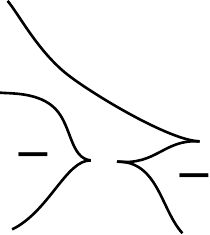},\\
&\embeddpdf{-4}{3}{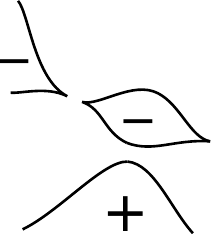}+\embeddpdf{-4}{3}{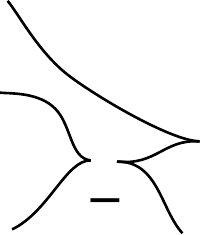}+
\embeddpdf{-4}{3}{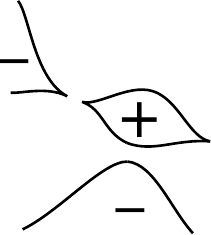},\quad \embeddpdf{-4}{3}{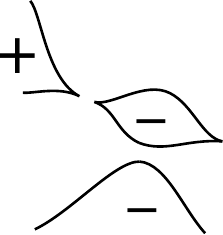}+
\embeddpdf{-4}{3}{Lr2ab-.pdf}+\embeddpdf{-4}{3}{Lr2ba-+-.pdf}, \quad
\embeddpdf{-4}{3}{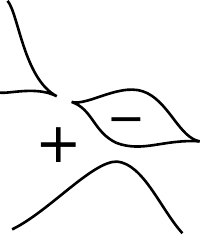}+\embeddpdf{-4}{3}{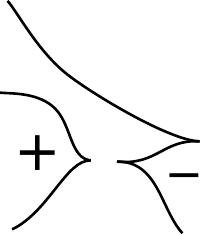}+
\embeddpdf{-4}{3}{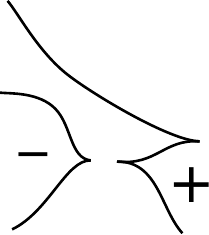}.
\end{flalign*}

 The complex $C_r\left(\embeddpdf{-2}{2}{Lr2for1.pdf}\right) $ is generated by all enhanced Legendrian states shown as follows:
 
 \begin{flalign*}
&\embeddpdf{-4}{3}{Lr2ab--.pdf},\quad \embeddpdf{-4}{3}{Lr2ab+-.pdf}+\embeddpdf{-4}{3}{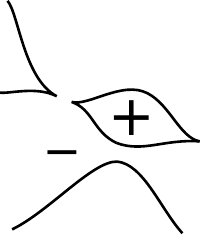},\quad\embeddpdf{-4}{3}{Lr2ab-+.pdf}+\embeddpdf{-4}{3}{Lr2ba+-.pdf}, \quad \embeddpdf{-4}{3}{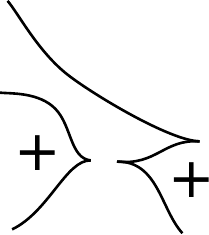}+\embeddpdf{-4}{3}{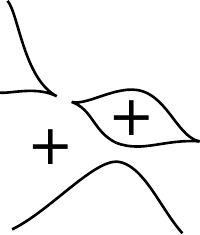},\\ &\embeddpdf{-4}{3}{Lr2ab-.pdf}+\embeddpdf{-4}{3}{Lr2ba-+-.pdf},\quad\embeddpdf{-4}{3}{Lr2ab+.pdf}+\embeddpdf{-4}{3}{Lr2ba++-.pdf}+\embeddpdf{-4}{3}{Lr2ba-++.pdf}.
\end{flalign*}

A term $S$ in any generator from the above shown list involves two strands. These two stands can be part of one loop or two loops in $S$. To make this distinction we give two separate labels for these strands if they are part of two different loops in $S$ or only one label otherwise. For example, the strands belong to two loops in $\embeddpdf{-2}{2}{Lr2ab-+.pdf}$ with labels $-$ and $+$, and part of one loop in $\embeddpdf{-2}{2}{Lr2ab+.pdf}$ with label $+$.

The Legendrian Khovanov complex $C\left(\embeddpdf{-2}{2}{Lr2for1.pdf}\right)$ deformation retracts onto Legendrian Khovanov  subcomplex $C_r\left(\embeddpdf{-2}{2}{Lr2for1.pdf}\right)$. The retraction map  $\rho_2:C\left(\embeddpdf{-2}{2}{Lr2for1.pdf}\right)\rightarrow C_r\left(\embeddpdf{-2}{2}{Lr2for1.pdf}\right)$ and chain-homotopy map $h_2:C\left(\embeddpdf{-2}{2}{Lr2for1.pdf}\right)\rightarrow C\left(\embeddpdf{-2}{2}{Lr2for1.pdf}\right)$ between $i\circ \rho_2$ and $id :C\left(\embeddpdf{-2}{2}{Lr2for1.pdf}\right)\rightarrow C\left(\embeddpdf{-2}{2}{Lr2for1.pdf}\right)$ on each generator are given in Table \ref{tab:tab1} in Appendix.

Each generator of $C_r\left(\embeddpdf{-2}{2}{Lr2for1.pdf}\right)$ involves a term which is locally of the form $\embeddpdf{-2}{2}{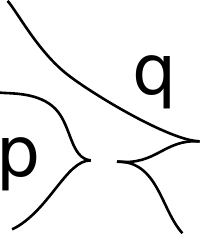}$, with labels $p,q\in \{+,-\}$. Then, each generator $\gamma$ can be written as $\embeddpdf{-2}{2}{Lr2forpq.pdf}+Q$, where $Q$ is the sum of the rest of the terms in $\gamma$. The isomorphism $\Phi_2: C_r\left(\embeddpdf{-2}{2}{Lr2for1.pdf}\right)\rightarrow C\left(\embeddpdf{-2}{2}{Lr2back1.pdf}\right)$ is given by 

$$\Phi_2(\gamma)=\Phi_2\left(\embeddpdf{-2}{2}{Lr2forpq.pdf}+Q\right)=\embeddpdf{-2}{2}{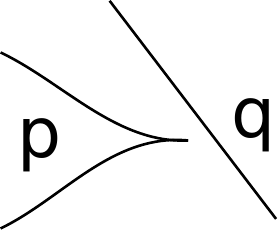}.$$ 
We notice that $\Phi_2\circ\partial=\partial\circ\Phi_2$ holds. Hence, the Legendrian Khovanov homology is invariant under LR2.

Finally, for the invariance of Legedrian Khovanov homology under LR3, we show that the Legendrian Khovanov complex for  $\embeddpdf{-2}{2}{LR3for.pdf}$ deformation retracts onto the Legendrian Khovanov subcomplex, which we denote by
$C_r\left(\embeddpdf{-2}{2}{LR3for.pdf}\right),$ with the retraction map given by $\rho_3 :C\left(\embeddpdf{-2}{2}{LR3for.pdf}\right)\rightarrow C_r\left(\embeddpdf{-2}{2}{LR3for.pdf}\right)$. 
The details are given in Table \ref{tab:tab2},\ref{tab:tab3} and \ref{tab:tab4} in Appendix. As, $\rho_3^2=\rho_3$, that is, $\rho_3 $ is a projection map. Therefore, we naturally get a decomposition $$C\left(\embeddpdf{-2}{2}{LR3for.pdf}\right)=C_r\left(\embeddpdf{-2}{2}{LR3for.pdf}\right)\bigoplus C_{cr}\left(\embeddpdf{-2}{2}{LR3for.pdf}\right)$$  where $C_{cr}\left(\embeddpdf{-2}{2}{LR3for.pdf}\right)$ is the kernel of $\rho_3$ and the generators of $C_r\left(\embeddpdf{-2}{2}{LR3for.pdf}\right)$ are shown as follows:

\begin{flalign*}
\embeddpdf{-6}{4}{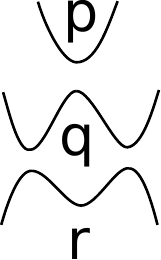},\quad \embeddpdf{-6}{4}{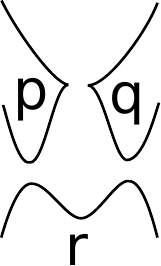},\quad \embeddpdf{-6}{4}{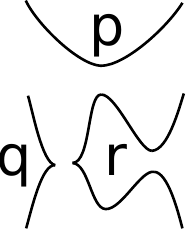}+\embeddpdf{-6}{4}{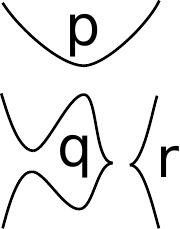},\quad \rho_3\left(\embeddpdf{-6}{4}{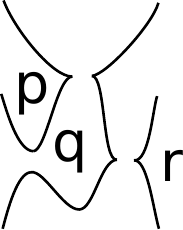}\right), \quad \rho_3\left(\embeddpdf{-6}{4}{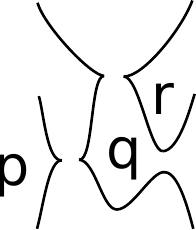}\right)
\end{flalign*}
 The homotopy connecting $i\circ \rho_3$ and $id: C\left(\embeddpdf{-2}{2}{LR3for.pdf}\right) \rightarrow C\left(\embeddpdf{-2}{2}{LR3for.pdf}\right),$ is given by the map $h_3: C\left(\embeddpdf{-2}{2}{LR3for.pdf}\right) \rightarrow C\left(\embeddpdf{-2}{2}{LR3for.pdf}\right)$. The map $h_3$ on each generator is described in Table \ref{tab:tab2}, \ref{tab:tab3} and \ref{tab:tab4} given in Appendix. \\

Simlilarly, we can split $C\left(\embeddpdf{-2}{2}{LR3back.pdf}\right)$ into the direct sum of the Legendrian Khovanov subcomplexes $C_r\left(\embeddpdf{-2}{2}{LR3back.pdf}\right)$ and $C_{cr}\left(\embeddpdf{-2}{2}{LR3back.pdf}\right)$ using the retraction map $\rho_3': C\left(\embeddpdf{-2}{2}{LR3back.pdf}\right)\rightarrow C_r\left(\embeddpdf{-2}{2}{LR3back.pdf}\right)$. 
Again, the retraction map $\rho_3'$ and the homotopy $h_3'$ connecting $id :C\left(\embeddpdf{-2}{2}{LR3back.pdf}\right)\rightarrow C\left(\embeddpdf{-2}{2}{LR3back.pdf}\right)$ to $i\circ \rho_3'$ is given on each 
generator of $C\left(\embeddpdf{-2}{2}{LR3back.pdf}\right)$ in Table \ref{tab:tab5}, \ref{tab:tab6} and \ref{tab:tab7} in Appendix. The generators of  $C_r\left(\embeddpdf{-2}{2}{LR3back.pdf}\right)$ are shown as follows:
\begin{flalign*}
\embeddpdf{-6}{4}{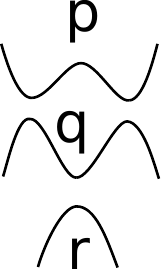},\quad\embeddpdf{-5}{4}{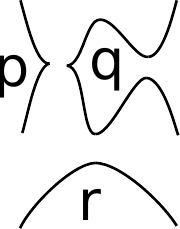}+\embeddpdf{-6}{4}{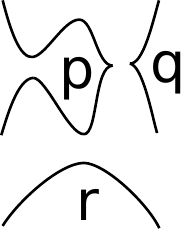},\quad \embeddpdf{-6}{4}{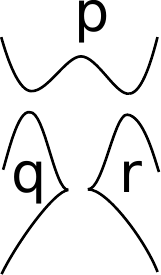}, \quad\rho_3'\left(\embeddpdf{-6}{4}{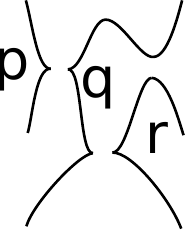}\right),\quad\rho_3'\left(\embeddpdf{-5}{4}{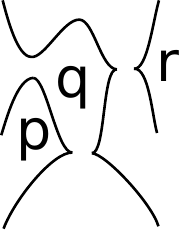}\right).
\end{flalign*}

The isomorphism $\Phi_3:C_r\left(\embeddpdf{-3}{2}{LR3back.pdf}\right)\rightarrow C_r\left(\embeddpdf{-3}{2}{LR3for.pdf}\right)$, is given on generators by
\begin{flalign*}
&\Phi_3\left(\embeddpdf{-6}{4}{Lr3aaapqr.pdf}\right)=\embeddpdf{-6}{4}{Lr3AAApqr.pdf},\quad \Phi_3\left(\embeddpdf{-5}{4}{Lr3baapqr.pdf}+\embeddpdf{-6}{4}{Lr3aabpqr.pdf}\right)=\embeddpdf{-6}{4}{Lr3ABApqr.pdf}, \quad\Phi_3\left(\embeddpdf{-6}{4}{Lr3abapqr.pdf}\right)=\embeddpdf{-6}{4}{Lr3BAApqr.pdf}+\embeddpdf{-6}{4}{Lr3AABpqr.pdf} \\
&\Phi_3\left(\rho_3'\left(\embeddpdf{-6}{4}{Lr3bbapqr.pdf}\right)\right)=\rho_3\left(\embeddpdf{-6}{4}{Lr3ABBpqr.pdf}\right),\quad\Phi_3\left(\rho_3'\left(\embeddpdf{-5}{4}{Lr3abbpqr.pdf}\right)\right)=\rho_3\left(\embeddpdf{-6}{4}{Lr3BBApqr.pdf}\right).
\end{flalign*}

It is easy to see that $\Phi_3 $ commutes with the boundary maps. 
Hence, Legendrian Khovanov homology is invariant under LR3. 
\end{proof}

We conclude by noting that the Jones polynomial for Legendrian knot admits categorification using the homology introduced in this section.

\begin{proof}[Proof of Theorem \ref{thm4}]
The graded Euler characteristic of $H_{i,k,j}(K)$ is given by the following:
\begin{align*}
\sum_{i,k,j}(-1)^ir^kq^jdim(C_{i,k,j}(K_F))&=\sum_{S \in \mathcal{ELS}(K_F)}(-1)^{i(S)}q^{j(S)}r^{k(S)} \\
&=P_K(A,r). \hspace{20pt} \left(\text{ using } \eqref{eq5} \right). &&\qedhere
\end{align*}
\end{proof}
 
\section{The Legendrian Khovanov homology with integer coefficients}
 
To define the Legendrian Khovanov homology with integer coefficients, we refine the Legendrian Khovanov complex by giving an orientation on the enhanced Legendrian state in a manner such that the number of generators of the Legendrian
Khovanov complex does not change. Define an \emph{oriented} enhanced Legendrian state as a pair $(S,l=x_1<x_2<\cdots <x_n)$, where $S$ is an enhanced Legendrian state and $l$ is an ordered sequence of crossings  which went under $B$-resolutions to obtain $S$ from $K_F$. We say that $(S,l)=(S,l')$, i.e. they have same orientation, if $l$ and $l'$ differ by an even permutation. Now
we  define the Legendrian Khovanov complex with integer coefficients $C_{\mathbb{Z}}(K_F)$ as the abelian group generated by all oriented enhanced Legendrian states $(S,l)$ over $\mathbb{Z}$. Also, we assume
$-(S,l)=(S,p(l))$, where $p$ is an odd permutation. Then the number of generators for $C(K_F)$ and $C_{\mathbb{Z}}(K_F)$ are equal. 

\begin{definition}
An oriented enhanced Legendrian state $(T,l')$ is  said to be \emph{incident} with $(S,l)$, if $T$ is incident with $S$ and $(T,l')=(T,l<x(S:T))$. The notation $l<x(S:T) $ denotes the ordering of crossings in $l$ followed by $x(S:T) $.
\end{definition}

\begin{definition}
The incidence number $((S,l):(T,l'))$ is $1$ if $(T,l)$ is incident with $(S,l')$ and $0$ otherwise.
\end{definition}

Now we are ready to define the boundary map $\partial_{\mathbb{Z}}:C_\mathbb{Z} (K_F)\rightarrow C_{\mathbb{Z}}(K_F)$. The boundary map $\partial_{\mathbb{Z}}$ simply 
sends an oriented Legendrian enhanced state to the sum of all the oriented enhanced Legendrian states which are incident with it. Let $\mathcal{OELS}(K_F) $ denote the set of all oriented enhanced Legendrian states obtained from $K_F $. Define the boundary map as follows:
\begin{align*}
\partial_{\mathbb{Z}}\left((S:l)\right) :=\sum_{(T,l') \in \mathcal{OELS}(K_F) }\left((S,l):(T,l')\right)(T,l').    
\end{align*}

\begin{lemma}
$\partial_{\mathbb{Z}}^2=0$.
\end{lemma}
\begin{proof}
The proof of this lemma is similar to Lemma \ref{b2}. Instead of proving that every term appears even number of times, here it can be shown that the terms appear in pairs with opposite orientations. For example, in Case 1 shown in Figure \ref{fig:case1}, we have 
\begin{flalign*}
\left(\embeddpdf{-2}{1.5}{incase1-.pdf},l\right)\longrightarrow \left(\embeddpdf{-2}{1.5}{incase1--1.pdf},l<x\right)+\left(\embeddpdf{-2}{1.5}{incase1--.pdf},l<y\right)\\
\longrightarrow \left(\embeddpdf{-2}{1.5}{incase1---.pdf},l<x<y\right)+\left(\embeddpdf{-2}{1.5}{incase1---.pdf},l<y<x\right)=0.
\end{flalign*} 

All other cases are just routine calculations.

The proof for the invariance of the Legendrian Khovanov homology with $\mathbb{Z}$ coefficients under the Legendrian Reidemeister moves is similar to the one given in Theorem \ref{thm2}. In the proof of Theorem 
\ref{thm2}, we showed that the Legendrian Khovanov complex of the front projection obtained after applying the the LR move retracts onto a Legendrian Khovanov subcomplex which is isomorphic to the Legendrian Khovanov complex of the projection before the LR move. If we consider each generator used in the proof of Theorem \ref{thm2} with the orientation given as below then it will give the proof of invariance of the Legendrian Khovanov homology with $\mathbb{Z}$-coefficients. If we give an orientation on each generator of $C(K_F)$ we get the generators of $C_{\mathbb{Z}}(K_F)$. 

The generators of $C_{\mathbb{Z}}\left(\embeddpdf{-2}{2}{Lr11back.pdf}\right)$ and $C_{\mathbb{Z}}\left(\embeddpdf{-2}{2}{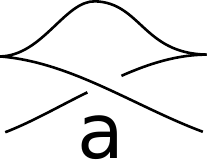}\right)$ are locally shown as 
$\left(\embeddpdf{-2}{2}{Lr11back.pdf},l\right)$ and $\left(\embeddpdf{-2}{2}{Lr1a.pdf},l\right),\left(\embeddpdf{-2}{2}{Lr1b.pdf},l<a\right)$ respectively.

The generators of $C_{\mathbb{Z}}\left(\embeddpdf{-4}{3}{Lr2back1.pdf}\right)$ and  $C_{\mathbb{Z}}\left(\embeddpdf{-4}{3}{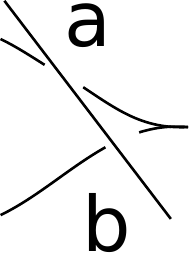}\right)$ are locally shown as 
$\left(\embeddpdf{-4}{3}{Lr2back1.pdf},l\right)$ and $\left(\embeddpdf{-4}{3}{Lr2aa.pdf},l\right),\left(\embeddpdf{-4}{3}{Lr2ab.pdf},l<b\right),\left(\embeddpdf{-4}{3}{Lr2ba.pdf},l<a\right),\left(\embeddpdf{-4}{3}{Lr2bb.pdf},l<a<b\right)$ respectively.

The generators of $C_{\mathbb{Z}}\left(\embeddpdf{-4}{3}{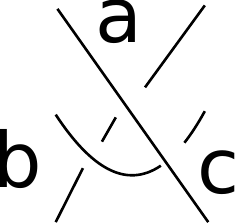}\right)$ are locally given as 
\begin{flalign*}
&\left(\embeddpdf{-4}{3}{Lr3AAA.pdf},l\right),\left(\embeddpdf{-4}{3}{LR3AAB.pdf},l<c\right),\left(\embeddpdf{-4}{3}{LR3ABA.pdf},l<a\right),\left(\embeddpdf{-4}{3}{LR3BAA.pdf},l<b\right),\left(\embeddpdf{-4}{3}{LR3ABB.pdf},l<a<c\right),\\
&\left(\embeddpdf{-4}{3}{LR3BAB.pdf},l<b<c\right),\left(\embeddpdf{-4}{3}{LR3BBA.pdf},l<b<a\right),\left(\embeddpdf{-4}{3}{LR3BBB.pdf},l<a<b<c\right).
\end{flalign*}

The generators of $C_{\mathbb{Z}}\left(\embeddpdf{-4}{3}{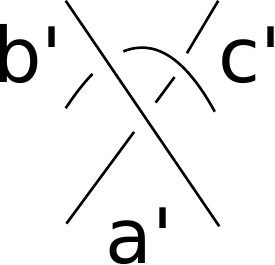}\right)$ are locally given as 
\begin{align*}
&\left(\embeddpdf{-4}{3}{Lr3aaa.pdf},l\right),\left(\embeddpdf{-4}{3}{LR3aab.pdf},l<c'\right),\left(\embeddpdf{-4}{3}{LR3aba.pdf},l<a'\right),\left(\embeddpdf{-4}{3}{LR3baa.pdf},l<b'\right),\left(\embeddpdf{-4}{3}{LR3abb.pdf},l<a'<c'\right),\\
&\left(\embeddpdf{-4}{3}{LR3bab.pdf},l'<b'<c'\right),\left(\embeddpdf{-4}{3}{LR3bba.pdf},l<b'<a'\right),\left(\embeddpdf{-4}{3}{LR3bbb.pdf},l<a'<b'<c'\right). &&\qedhere    
\end{align*}
\end{proof}
\section{Strengths and limitations of Legendrian Khovanov homology} 
In the previous sections, we obtained a formulation of Legendrian Jones polynomial and its categorification in such a way that we can recover the Jones polynomial and Khovanov homology respectively of the underlying knot.  It is, therefore, compelling to compare the new invariants with the existing ones. We begin by observing the following.
\begin{theorem}\label{thm6}
Let $K$ be a Legendrian knot, then $H_{i,k,j}(K)=H_{i,j-tb(K),j}(K)$ for all $i,j,k \in \mathbb{Z}$.
\end{theorem}

\begin{proof}
Let $S$ be any enhanced Legendrian state of $K_F$, where $K_F$ is front projection of $K$. Then,
 \begin{flalign*}
 j(S)-k(S)&=\frac{3\omega-\sigma(S)}{2}-\frac{c(S)}{2}+l(K_F)\\
 &= \frac{\omega-\sigma(S)}{2}+\omega-\frac{c(K_F)}{2}-B(S)+l(K_F)\\
 &=\frac{\omega-A(S)+B(S)-2B(S)+2l(K_F)}{2}+tb(K)\\
& =tb(K)
 \end{flalign*}
 Therefore, $C_{i,k,j}=C_{i,j-tb(K),j}$ for all $i,k,j \in \mathbb{Z}$. Thus, we have $H_{i,k,j}(K)=H_{i,j-tb(K),j}(K)$.
\end{proof}

\begin{proof}[Proof of Theorem \ref{thm3}]

By Theorem \ref{thm6}, we have $k(S)=j(S)-tb(K)$. Hence, we see that while the grading $k$ gives a new grading on the Legendrian Khovanov complex (which also captures the
geometry of the knot), Khovanov complex for the Legendrian knot and the underlying smooth knot are equal. That is, $C_{i,k,j}(K)=C_{i,j}(K)$ for all $i,k,j \in \mathbb{Z}.$ Since
$\partial$ preserves $j$ and $k$, it reduces to the boundary map for the Khovanov complex for the corresponding smooth knot after dropping the grading $k$ and
we get $ker(\partial_{i,k,j})=ker(\partial_{i,j})$ and $Im(\partial_{i,k,j})=Im(\partial_{i,j})$.
Hence, we have $H_{i,j}(K)=H_{i,j-tb(K),j}(K)$.
\end{proof}

 \begin{proof}[Proof of Corollary \ref{cor1}]
Let $K$ and $K'$ be smoothly isotopic, we get $H_{i,j}(K)=H_{i,j}(K')$ for all $i,j \in \mathbb{Z}$. Using  Theorem \ref{thm3} and Theorem \ref{thm6}, $H_{i,j-tb(K),j}(K)=H_{i,j-tb(K'),j}(K')$ if and only if $tb(K)=tb(K').$
 \end{proof}
 
 \begin{remark}
Theorem \ref{thm3} gives us an algebraic interpretation of the Thurston-Bennequin number (which counts the rotations of contact planes along the Legendrian knot) as the \emph{grade shift} in the newly introduced grading of the Legendrian Khovanov homology. As a consequence, stabilization of a Legendrian knot will lead to shifting the $k$-grading of the Legendrian Khovanov homology of the resultant knot.
\end{remark}

It is worth noting that the Legendrian Khovanov Homology fails to distinguish the following two Legendrian knots aka “Chekanov knots”  shown in Figure \ref{fig:chekanovknots}. These knots are not Legendrian isotopic as shown in \cite{Etnyre,cheka2} using the Chekanov-Eliashberg DGA \cite{cheka1,cheka2}. It is known that the knots are smoothly isotopic and their Thurston–Bennequin invariants agree. Therefore, by Corollary \ref{cor1}, they have the same Legendrian Khovanov homology groups.

\begin{figure}[ht!]
    \centering
    \includegraphics{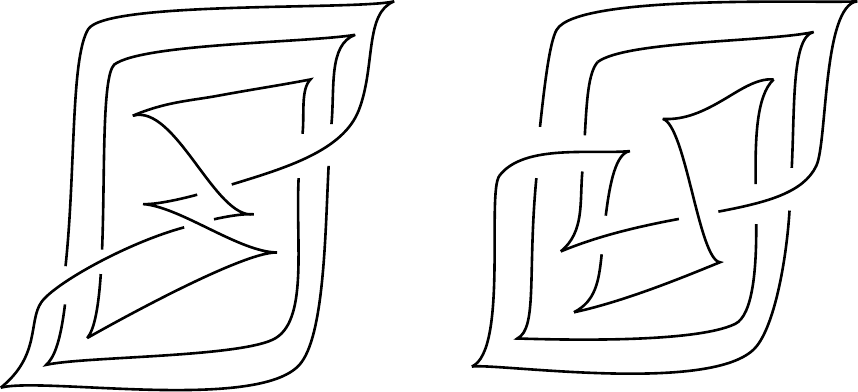}
    \caption{Front projections of the Chekanov knots}
    \label{fig:chekanovknots}
\end{figure}

The Legendrian Jones polynomial for these knots is given by
\begin{align*}
 A^{-18}r^5&(A^6\delta^5+6A^4r\delta^4+14A^2r^2\delta^3+A^2r^2\delta^5+16r^3\delta^2\\&+4r^3\delta^4+8A^{-2}r^4\delta+7A^{-2}r^4\delta^3+6A^{-4}r^5\delta^2+A^{-6}r^6\delta) 
\end{align*}
where $\delta=-A^2r^{-1}-A^{-2}r$.

\section{Appendix}
In this section, we provide the details of retraction maps and chain homotopies that are required for showing invariance of Legendrian Khovanov homology.

\begin{table}[h!]
    \centering
    \[
\begin{tblr}{|c|c|c|c|c|c|}
\hline
S & h_2(S)& \rho_2(S)&S&h_2(S)&\rho_2(S)\\
\hline
\embeddpdf{-4}{3}{Lr2aa.pdf}&0&0&\embeddpdf{-4}{3}{Lr2ba+.pdf}&0&0\\
\hline
    \embeddpdf{-4}{3}{Lr2ab--.pdf}&0&\embeddpdf{-4}{3}{Lr2ab--.pdf}&\embeddpdf{-4}{3}{Lr2ab+-.pdf}&0&\embeddpdf{-4}{3}{Lr2ab+-.pdf}+\embeddpdf{-4}{3}{Lr2ba+-.pdf}\\
    \hline
    \embeddpdf{-4}{3}{Lr2ab-+.pdf}&0& \embeddpdf{-4}{3}{Lr2ab-+.pdf}+\embeddpdf{-4}{3}{Lr2ba+-.pdf}&
\embeddpdf{-4}{3}{Lr2ab++.pdf}&0& \embeddpdf{-4}{3}{Lr2ab++.pdf}+\embeddpdf{-4}{3}{Lr2ba++.pdf},\\
    \hline
    \embeddpdf{-4}{3}{Lr2ab-.pdf}&0& \embeddpdf{-4}{3}{Lr2ab-.pdf}+\embeddpdf{-4}{3}{Lr2ba-+-.pdf}
&\embeddpdf{-4}{3}{Lr2ab+.pdf}&0& \embeddpdf{-4}{3}{Lr2ab+.pdf}+\embeddpdf{-4}{3}{Lr2ba++-.pdf}+\embeddpdf{-4}{3}{Lr2ba-++.pdf},\\
\hline
\embeddpdf{-4}{3}{Lr2ba+-+.pdf}&\embeddpdf{-4}{3}{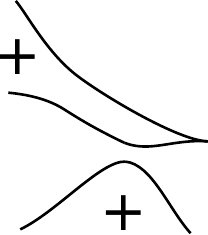}& -\embeddpdf{-4}{3}{Lr2ba-++.pdf}-\embeddpdf{-4}{3}{Lr2ab+.pdf}-\embeddpdf{-4}{3}{Lr2ba++-.pdf}
&\embeddpdf{-4}{3}{Lr2ba--+.pdf}&\embeddpdf{-4}{3}{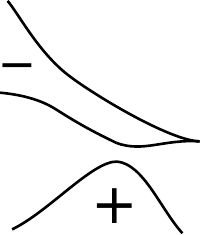}& -\embeddpdf{-4}{3}{Lr2ab-.pdf}-\embeddpdf{-4}{3}{Lr2ba-+-.pdf},\\
\hline
\embeddpdf{-4}{3}{Lr2ba+--.pdf}&\embeddpdf{-4}{3}{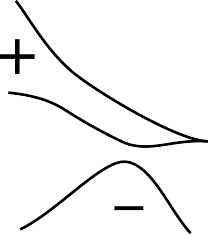}& -\embeddpdf{-4}{3}{Lr2ab-.pdf}-\embeddpdf{-4}{3}{Lr2ba-+-.pdf}
&\embeddpdf{-4}{3}{Lr2ba-+.pdf}&\embeddpdf{-4}{3}{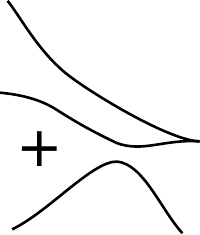} &-\embeddpdf{-4}{3}{Lr2ab+-.pdf}-\embeddpdf{-4}{3}{Lr2ab-+.pdf},\\
\hline
\embeddpdf{-4}{3}{Lr2ba--.pdf}&\embeddpdf{-4}{3}{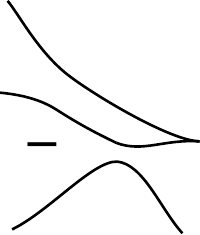}& -\embeddpdf{-4}{3}{Lr2ab--.pdf}&\embeddpdf{-4}{3}{Lr2ba---.pdf}&\embeddpdf{-4}{3}{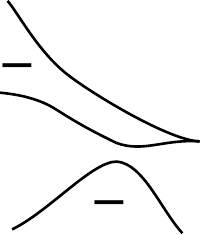}&0\\
\hline
\embeddpdf{-4}{3}{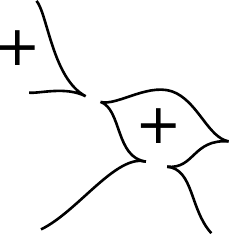}  & \embeddpdf{-4}{3}{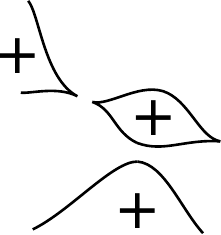}&0&
\embeddpdf{-4}{3}{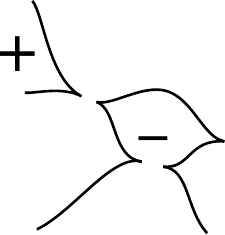}  & \embeddpdf{-4}{3}{Lr2ba++-.pdf}&0\\
\hline
\embeddpdf{-4}{3}{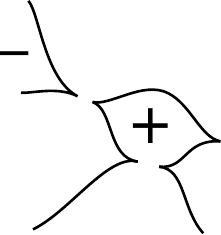}& \embeddpdf{-4}{3}{Lr2ba-++.pdf}&0&\embeddpdf{-4}{3}{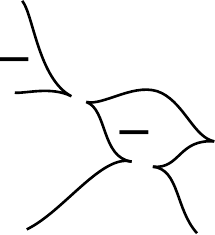} &\embeddpdf{-4}{3}{Lr2ba-+-.pdf}&0\\
\hline
\embeddpdf{-4}{3}{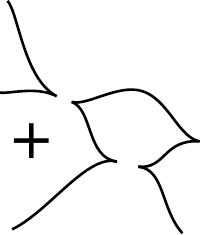}  &\embeddpdf{-4}{3}{Lr2ba++.pdf} &0& 
\embeddpdf{-4}{3}{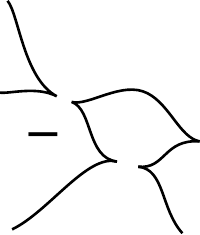}  & \embeddpdf{-4}{3}{Lr2ba+-.pdf} &0\\
\hline
    \end{tblr}
\]
\caption{}
    \label{tab:tab1}
\end{table}
In the local region where $LR3$ is applied, we get three strands after resolving the crossings. In Table \ref{tab:tab2}, \ref{tab:tab3}, \ref{tab:tab4}, \ref{tab:tab5}, \ref{tab:tab6} and \ref{tab:tab7}, we have used dotted lines to describe all the possibilities of how these strands might be connected in  a particular generator of $C\left(\embeddpdf{-2}{2}{LR3back.pdf}\right)$ and $C\left(\embeddpdf{-2}{2}{LR3for.pdf}\right)$. 

\begin{table}[ht!]
    \[
\begin{tblr}{|c|c|c|c|c|c|}
\hline
S & h_3(S)& \rho_3(S)&S&h_3(S)&\rho_3(S)\\
\hline
    \embeddpdf{-3.8}{3.8}{LR3AAA.pdf} & 0 &\embeddpdf{-3.8}{3.8}{LR3AAA.pdf}& \embeddpdf{-3.8}{3.8}{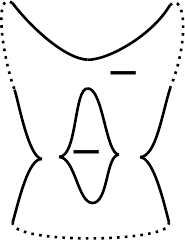} &-\embeddpdf{-3.8}{3.8}{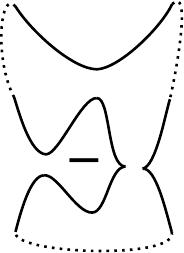}  & -\embeddpdf{-3.8}{3.8}{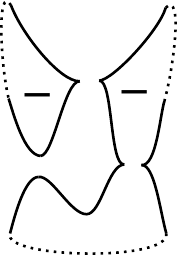} \\
    \hline
    \embeddpdf{-3.8}{3.8}{LR3AAB.pdf} & 0 & 0 & \embeddpdf{-3.8}{3.8}{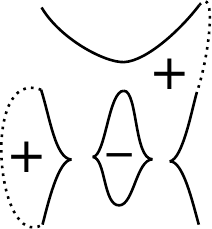} & -\embeddpdf{-3.8}{3.8}{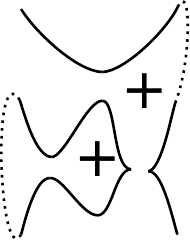} & -\embeddpdf{-3.8}{3.8}{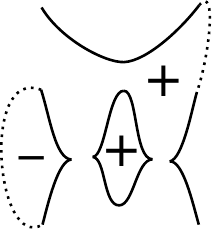}-\embeddpdf{-3.8}{3.8}{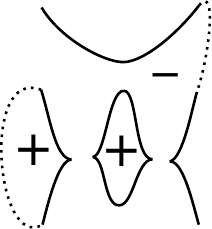}-\embeddpdf{-3.8}{3.8}{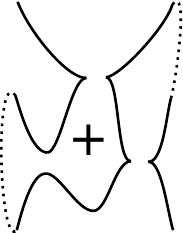}\\
    \hline
    \embeddpdf{-3.8}{3.8}{LR3BAA.pdf} & 0 & \embeddpdf{-3.8}{3.8}{LR3BAA.pdf}+\embeddpdf{-3.8}{3.8}{LR3AAB.pdf} &  \embeddpdf{-3.8}{3.8}{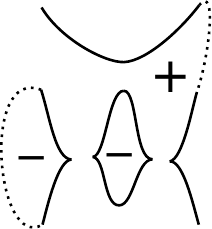} & -\embeddpdf{-3.8}{3.8}{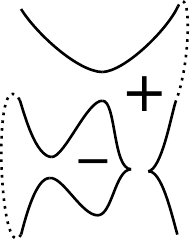} & -\embeddpdf{-3.8}{3.8}{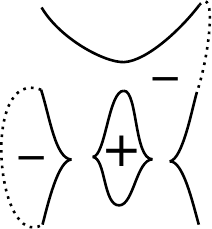}-\embeddpdf{-3.8}{3.8}{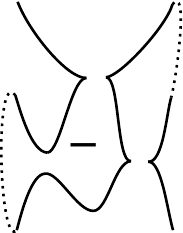} \\
    \hline
    \embeddpdf{-3.8}{3.8}{LR3ABA.pdf} & 0 & \embeddpdf{-3.8}{3.8}{LR3ABA.pdf} & \embeddpdf{-3.8}{3.8}{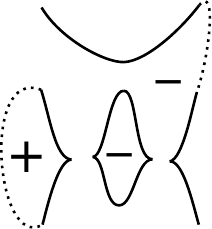} & -\embeddpdf{-3.8}{3.8}{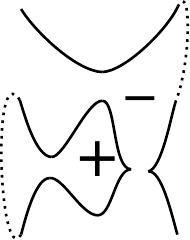} & -\embeddpdf{-3.8}{3.8}{LR3BAB+--1.pdf}-\embeddpdf{-3.8}{3.8}{LR3ABB-.pdf} \\
    \hline
    \embeddpdf{-3.8}{3.8}{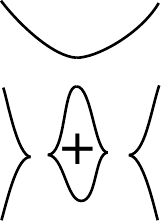} & 0 & 0 &\embeddpdf{-3.8}{3.8}{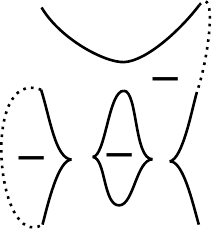} & -\embeddpdf{-3.8}{3.8}{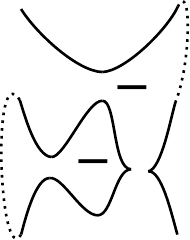} & 0\\
    \hline
    \embeddpdf{-3.8}{3.8}{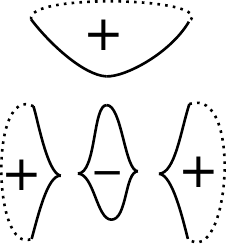} & -\embeddpdf{-3.8}{3.8}{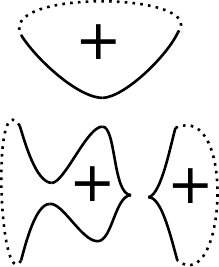} & -\embeddpdf{-3.8}{3.8}{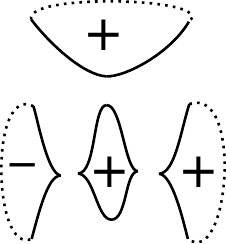}-\embeddpdf{-3.8}{3.8}{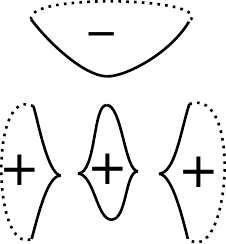}-\embeddpdf{-3.8}{3.8}{LR3ABB++}& \embeddpdf{-3.8}{3.8}{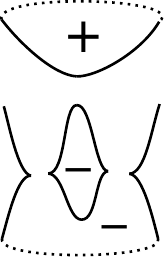} & -\embeddpdf{-3.8}{3.8}{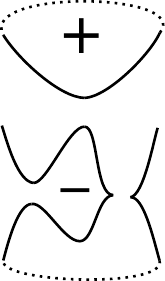}  & -\embeddpdf{-3.8}{3.8}{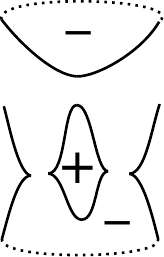}-\embeddpdf{-3.8}{3.8}{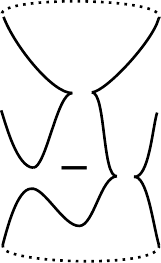} \\
    \hline
    \embeddpdf{-3.8}{3.8}{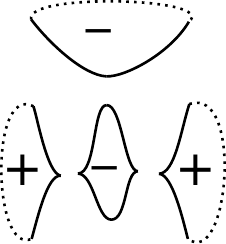} & -\embeddpdf{-3.8}{3.8}{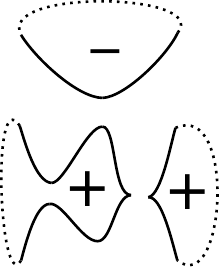} & -\embeddpdf{-3.8}{3.8}{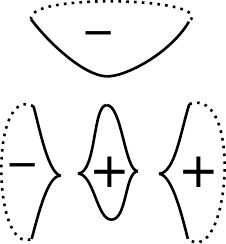}-\embeddpdf{-3.8}{3.8}{LR3ABB-+} &\embeddpdf{-3.8}{3.8}{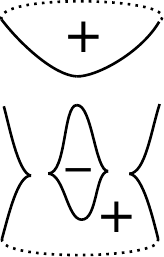} & -\embeddpdf{-3.8}{3.8}{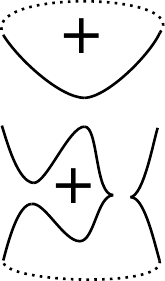}  & -\embeddpdf{-3.8}{3.8}{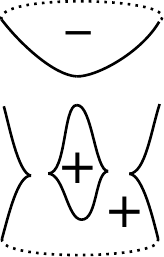}-\embeddpdf{-3.8}{3.8}{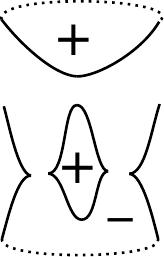}-\embeddpdf{-3.8}{3.8}{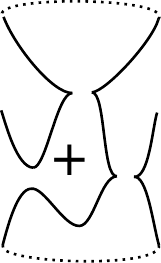}\\
    \hline
    \embeddpdf{-3.8}{3.8}{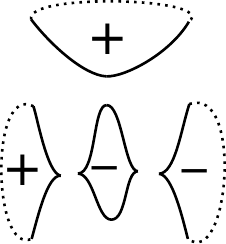} & -\embeddpdf{-3.8}{3.8}{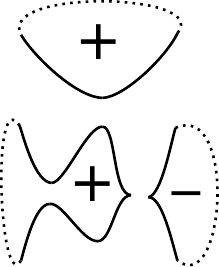} & -\embeddpdf{-3.8}{3.8}{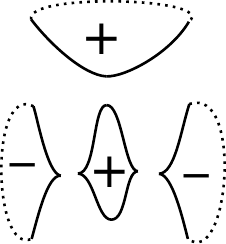}-\embeddpdf{-3.8}{3.8}{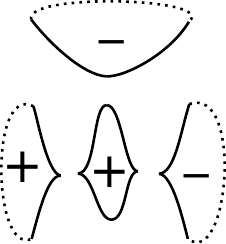}-\embeddpdf{-3.8}{3.8}{LR3ABB+-}&\embeddpdf{-3.8}{3.8}{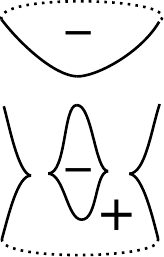} & -\embeddpdf{-3.8}{3.8}{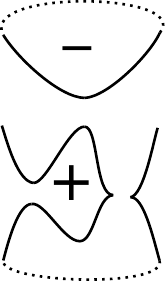}  & -\embeddpdf{-3.8}{3.8}{LR3BAB+--2.pdf}-\embeddpdf{-3.8}{3.8}{LR3ABB-1.pdf} \\
    \hline
    \embeddpdf{-3.8}{3.8}{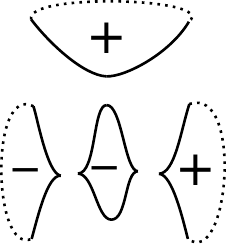} &- \embeddpdf{-3.8}{3.8}{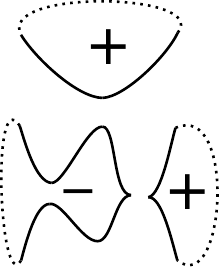} & -\embeddpdf{-3.8}{3.8}{LR3BAB+--+.pdf}-\embeddpdf{-3.8}{3.8}{LR3ABB-+} &\embeddpdf{-3.8}{3.8}{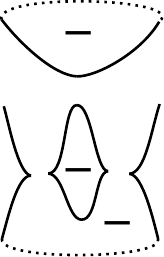} & -\embeddpdf{-3.8}{3.8}{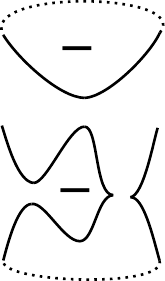}  & 0\\
    \hline
     \embeddpdf{-3.8}{3.8}{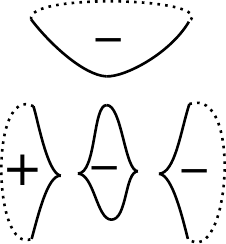} & -\embeddpdf{-3.8}{3.8}{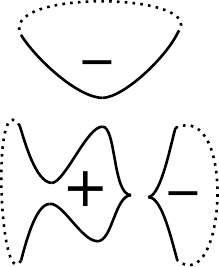} & -\embeddpdf{-3.8}{3.8}{LR3BAB+-+-.pdf}-\embeddpdf{-3.8}{3.8}{LR3ABB--}&\embeddpdf{-3.8}{3.8}{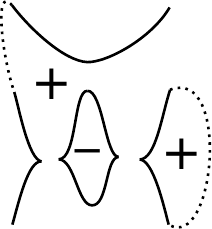} & -\embeddpdf{-3.8}{3.8}{LR3aab++1.pdf} & -\embeddpdf{-3.8}{3.8}{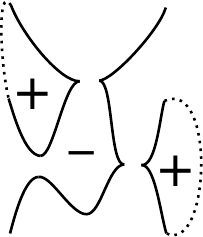}-\embeddpdf{-3.8}{3.8}{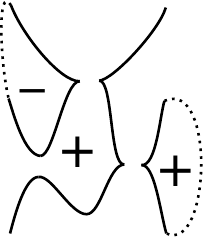} \\
    \hline
    \embeddpdf{-3.8}{3.8}{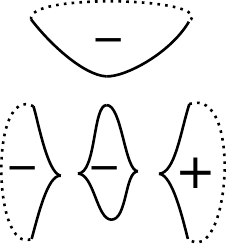} & -\embeddpdf{-3.8}{3.8}{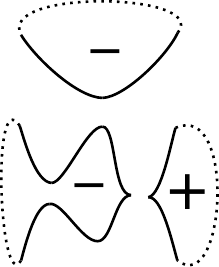} & 0 & \embeddpdf{-3.8}{3.8}{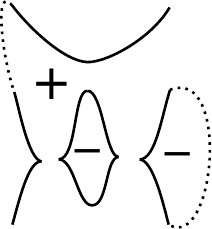} & -\embeddpdf{-3.8}{3.8}{LR3aab+-1.pdf} & -\embeddpdf{-3.8}{3.8}{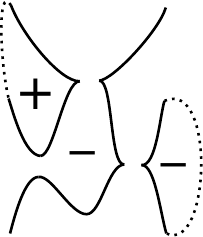}-\embeddpdf{-3.8}{3.8}{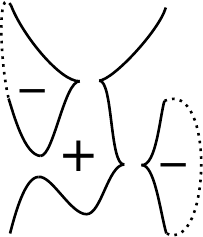}\\
    \hline
     \embeddpdf{-3.8}{3.8}{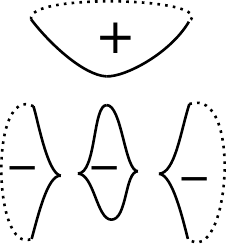} & -\embeddpdf{-3.8}{3.8}{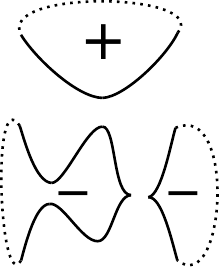} & -\embeddpdf{-3.8}{3.8}{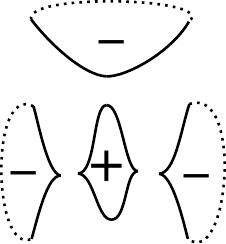}-\embeddpdf{-3.8}{3.8}{LR3ABB--} &\embeddpdf{-3.8}{3.8}{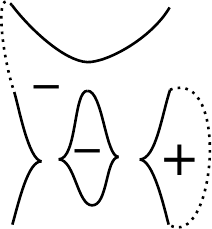} & -\embeddpdf{-3.8}{3.8}{LR3aab-+1.pdf} & -\embeddpdf{-3.8}{3.8}{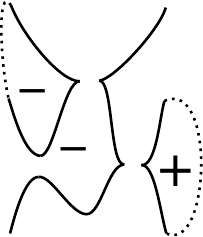}\\
    \hline
    \embeddpdf{-3.8}{3.8}{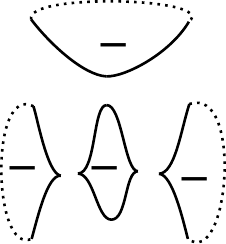} &-\embeddpdf{-3.8}{3.8}{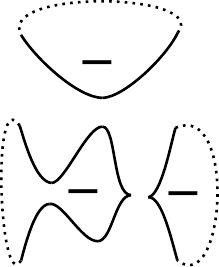}  & 0&\embeddpdf{-3.8}{3.8}{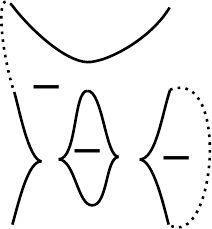} & -\embeddpdf{-3.8}{3.8}{LR3aab--1.pdf} & -\embeddpdf{-3.8}{3.8}{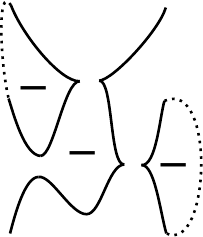} \\
    \hline
    \embeddpdf{-3.8}{3.8}{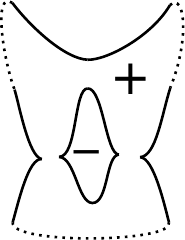} & -\embeddpdf{-3.8}{3.8}{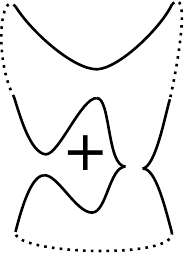} & -\embeddpdf{-3.8}{3.8}{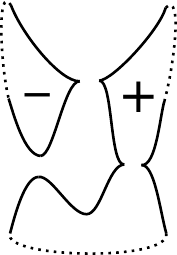}-\embeddpdf{-3.8}{3.8}{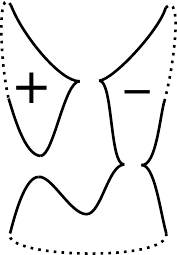}& \embeddpdf{-3.8}{3.8}{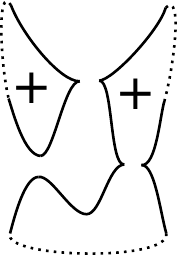} & 0 & \embeddpdf{-3.8}{3.8}{LR3ABB++1.pdf}+\embeddpdf{-3.8}{3.8}{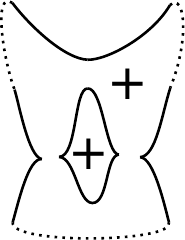}\\
    \hline
    \end{tblr}
\]
\caption{}
    \label{tab:tab2}
\end{table}

\begin{table}[ht!]
    \[
\begin{tblr}{|c|c|c|c|c|c|}
\hline
     S & h_3(S)& \rho_3(S)&S & h_3(S) & \rho_3(S)\\
    \hline
     \embeddpdf{-3.8}{3.8}{LR3ABB-+1.pdf} & 0 & \embeddpdf{-3.8}{3.8}{LR3ABB-+1.pdf} +\embeddpdf{-3.8}{3.8}{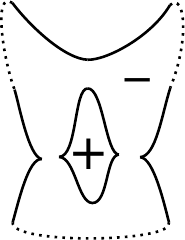}&\embeddpdf{-3.8}{3.8}{LR3ABB+--.pdf} & 0 & \embeddpdf{-3.8}{3.8}{LR3ABB+--.pdf}+\embeddpdf{-3.8}{3.8}{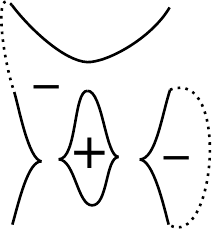}\\
    \hline
     \embeddpdf{-3.8}{3.8}{LR3ABB+-1.pdf} & 0 & \embeddpdf{-3.8}{3.8}{LR3ABB+-1.pdf}+\embeddpdf{-3.8}{3.8}{LR3BAB+-.pdf}&\embeddpdf{-3.8}{3.8}{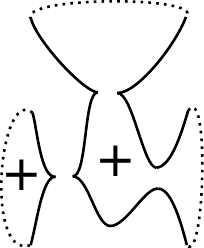} & 0 & \embeddpdf{-3.8}{3.8}{LR3BBA++.pdf}+\embeddpdf{-3.8}{3.8}{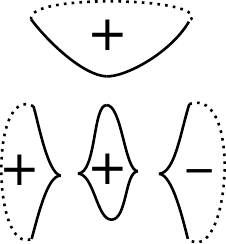}+\embeddpdf{-3.8}{3.8}{LR3BAB++-+.pdf} \\
     \hline
     \embeddpdf{-3.8}{3.8}{LR3ABB--1.pdf} &0  & \embeddpdf{-3.8}{3.8}{LR3ABB--1.pdf}& \embeddpdf{-3.8}{3.8}{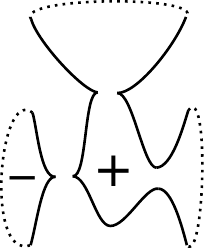} & 0 & \embeddpdf{-3.8}{3.8}{LR3BBA-+.pdf}+\embeddpdf{-3.8}{3.8}{LR3BAB+--+.pdf}+\embeddpdf{-3.8}{3.8}{LR3BAB+-+-.pdf} \\
    \hline
     \embeddpdf{-3.8}{3.8}{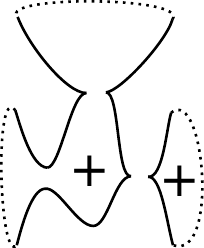} & 0 & \embeddpdf{-3.8}{3.8}{LR3ABB++.pdf}+\embeddpdf{-3.8}{3.8}{LR3BAB+-++.pdf}+\embeddpdf{-3.8}{3.8}{LR3BAB++-+.pdf} & \embeddpdf{-3.8}{3.8}{LR3ABB---.pdf} & 0 & \embeddpdf{-3.8}{3.8}{LR3ABB---.pdf}\\
     \hline
     \embeddpdf{-3.8}{3.8}{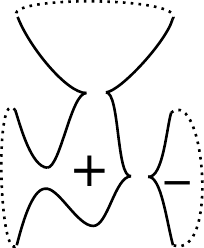} &0  & \embeddpdf{-3.8}{3.8}{LR3ABB+-.pdf}+\embeddpdf{-3.8}{3.8}{LR3BAB+-+-.pdf}+\embeddpdf{-3.8}{3.8}{LR3BAB++--.pdf}&
     \embeddpdf{-3.8}{3.8}{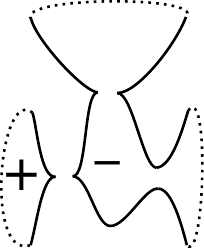} &0  & \embeddpdf{-3.8}{3.8}{LR3BBA+-.pdf}+\embeddpdf{-3.8}{3.8}{LR3BAB++--.pdf} \\
    \hline
     \embeddpdf{-3.8}{3.8}{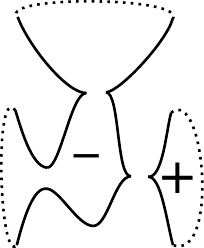} & 0 & \embeddpdf{-3.8}{3.8}{LR3ABB-+.pdf}+\embeddpdf{-3.8}{3.8}{LR3BAB+--+.pdf} &\embeddpdf{-3.8}{3.8}{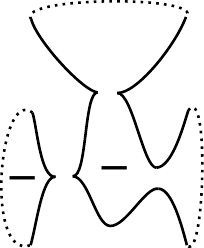} & 0 & \embeddpdf{-3.8}{3.8}{LR3BBA--.pdf} +\embeddpdf{-3.8}{3.8}{LR3BAB+---.pdf}\\
     \hline
     \embeddpdf{-3.8}{3.8}{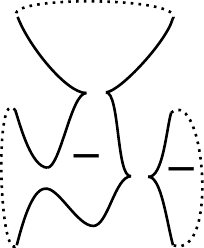} &0  & \embeddpdf{-3.8}{3.8}{LR3ABB--.pdf}+\embeddpdf{-3.8}{3.8}{LR3BAB+---.pdf}&\embeddpdf{-5}{5}{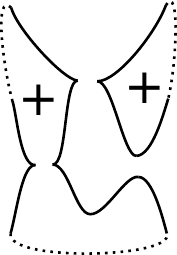} & 0 & \embeddpdf{-3.8}{3.8}{LR3BBA++1.pdf}+\embeddpdf{-3.8}{3.8}{LR3BAB++.pdf} \\
    \hline
     \embeddpdf{-3.8}{3.8}{LR3ABB+.pdf} & 0 & \embeddpdf{-3.8}{3.8}{LR3ABB+.pdf}+\embeddpdf{-3.8}{3.8}{LR3BAB++-1.pdf}+\embeddpdf{-3.8}{3.8}{LR3BAB+-+1.pdf}&\embeddpdf{-3.8}{3.8}{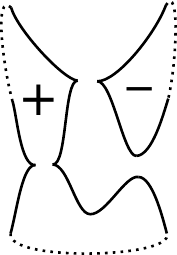} &0  & \embeddpdf{-3.8}{3.8}{LR3BBA+-1.pdf}+\embeddpdf{-3.8}{3.8}{LR3BAB+-.pdf} \\
     \hline
     \embeddpdf{-3.8}{3.8}{LR3ABB-.pdf} & 0 & \embeddpdf{-3.8}{3.8}{LR3ABB-.pdf}+\embeddpdf{-3.8}{3.8}{LR3BAB+--1.pdf} & \embeddpdf{-3.8}{3.8}{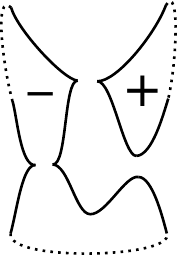} &  0& \embeddpdf{-3.8}{3.8}{LR3BBA-+1.pdf}+\embeddpdf{-3.8}{3.8}{LR3BAB+-.pdf} \\
     \hline
     \embeddpdf{-3.8}{3.8}{LR3ABB+1.pdf} & 0 & \embeddpdf{-3.8}{3.8}{LR3ABB+1.pdf}+\embeddpdf{-3.8}{3.8}{LR3BAB++-2.pdf}+\embeddpdf{-3.8}{3.8}{LR3BAB+-+2.pdf}&\embeddpdf{-3.8}{3.8}{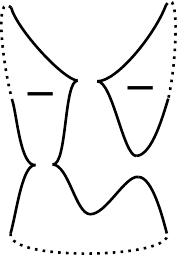} & 0 & \embeddpdf{-3.8}{3.8}{LR3BBA--1.pdf} \\
     \hline
      \embeddpdf{-3.8}{3.8}{LR3ABB-1.pdf} & 0 & \embeddpdf{-3.8}{3.8}{LR3ABB-1.pdf}+\embeddpdf{-3.8}{3.8}{LR3BAB+--2.pdf}& \embeddpdf{-3.8}{3.8}{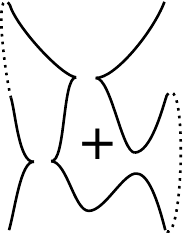} & 0 & \embeddpdf{-3.8}{3.8}{LR3BBA+.pdf}+\embeddpdf{-3.8}{3.8}{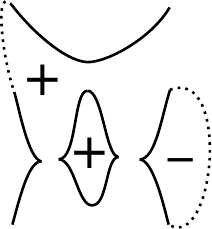}+\embeddpdf{-3.8}{3.8}{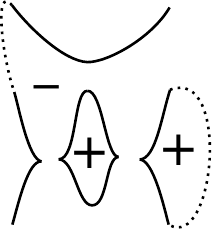} \\
     \hline
     \embeddpdf{-3.8}{3.8}{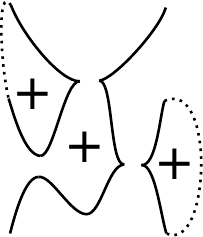} & 0 & \embeddpdf{-3.8}{3.8}{LR3ABB+++.pdf}+\embeddpdf{-3.8}{3.8}{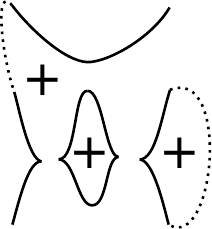}& \embeddpdf{-3.8}{3.8}{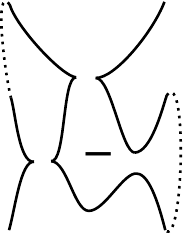} & 0 & \embeddpdf{-3.8}{3.8}{LR3BBA-.pdf}+\embeddpdf{-3.8}{3.8}{LR3BAB+--.pdf}  \\
     \hline
      \embeddpdf{-3.8}{3.8}{LR3ABB+-+.pdf} & 0 & \embeddpdf{-3.8}{3.8}{LR3ABB+-+.pdf}+\embeddpdf{-3.8}{3.8}{LR3BAB+-+.pdf} &\embeddpdf{-3.8}{3.8}{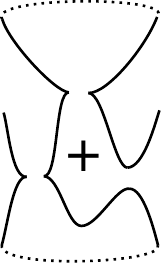} & 0 & \embeddpdf{-3.8}{3.8}{LR3BBA+1.pdf}+\embeddpdf{-3.8}{3.8}{LR3BAB++-2.pdf}+\embeddpdf{-3.8}{3.8}{LR3BAB+-+2.pdf}\\
     \hline
      \embeddpdf{-3.8}{3.8}{LR3ABB-++.pdf} &0  & \embeddpdf{-3.8}{3.8}{LR3ABB-++.pdf}+\embeddpdf{-3.8}{3.8}{LR3BAB+-+.pdf}&\embeddpdf{-3.8}{3.8}{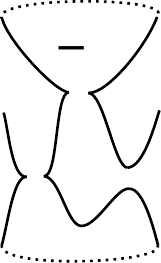} & 0 & \embeddpdf{-3.8}{3.8}{LR3BBA-1.pdf}+\embeddpdf{-3.8}{3.8}{LR3BAB+--2.pdf} \\
     \hline
     \embeddpdf{-3.8}{3.8}{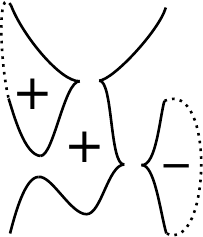} & 0 & \embeddpdf{-3.8}{3.8}{LR3ABB++-.pdf}+\embeddpdf{-3.8}{3.8}{LR3BAB++-.pdf}&\embeddpdf{-3.8}{3.8}{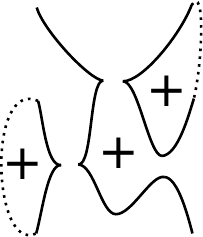} & 0 & \embeddpdf{-3.8}{3.8}{LR3BBA+++.pdf}+\embeddpdf{-3.8}{3.8}{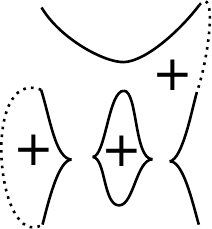} \\
     \hline
     \embeddpdf{-3.8}{3.8}{LR3ABB--+.pdf} & 0 & \embeddpdf{-3.8}{3.8}{LR3ABB--+.pdf}&\embeddpdf{-3.8}{3.8}{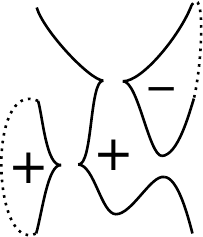} & 0 & \embeddpdf{-3.8}{3.8}{LR3BBA++-.pdf}+\embeddpdf{-3.8}{3.8}{LR3BAB++-1.pdf} \\
     \hline
\end{tblr}
\]
     \caption{}
    \label{tab:tab3}
\end{table}

\begin{table}[ht!]
    \centering
   \[
\begin{tblr}{|c|c|c|c|c|c|}
\hline
     S & h_3(S)& \rho_3(S)&S & h_3(S) & \rho_3(S)\\
\hline
      \hspace{5pt}\embeddpdf{-3.8}{3.8}{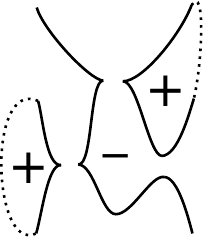}  \hspace{5pt}&  \hspace{15pt}0  \hspace{15pt}&  \hspace{8pt}\embeddpdf{-3.8}{3.8}{LR3BBA+-+.pdf}+\embeddpdf{-3.8}{3.8}{LR3BAB++-1.pdf} \hspace{8pt}&  \hspace{5pt}\embeddpdf{-3.8}{3.8}{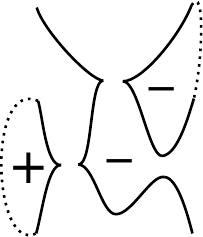} \hspace{5pt} & \hspace{5pt}0 \hspace{5pt}  & \hspace{8pt} \embeddpdf{-3.8}{3.8}{LR3BBA+--.pdf} \hspace{8pt} \\
     \hline
      \embeddpdf{-3.8}{3.8}{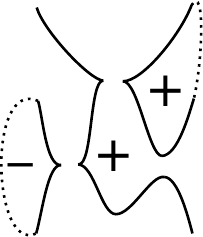} & 0 & \embeddpdf{-3.8}{3.8}{LR3BBA-++.pdf}+\embeddpdf{-3.8}{3.8}{LR3BAB+-+1.pdf} & \embeddpdf{-3.8}{3.8}{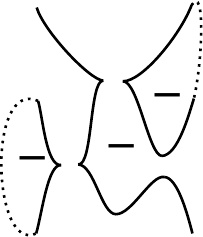} & 0 & \embeddpdf{-3.8}{3.8}{LR3BBA---.pdf}\\
     \hline
     \embeddpdf{-3.8}{3.8}{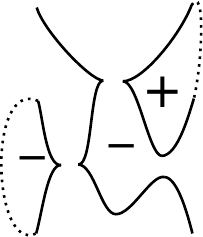} &  0& \embeddpdf{-3.8}{3.8}{LR3BBA--+.pdf}+\embeddpdf{-3.8}{3.8}{LR3BAB+--1.pdf}&  \embeddpdf{-3.8}{3.8}{LR3ABB-+-.pdf}&0 & \embeddpdf{-3.8}{3.8}{LR3ABB-+-.pdf}+\embeddpdf{-3.8}{3.8}{LR3BAB+--.pdf}\\
     \hline
      \embeddpdf{-3.8}{3.8}{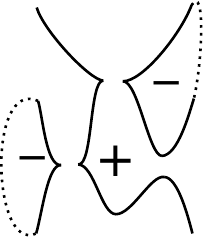} & 0 & \embeddpdf{-3.8}{3.8}{LR3BBA-+-.pdf}+\embeddpdf{-3.8}{3.8}{LR3BAB+--1.pdf}& \embeddpdf{-3.8}{3.8}{LR3BBB.pdf} &  \embeddpdf{-3.8}{3.8}{LR3BAB+.pdf}& 0 \\
     \hline
\end{tblr}
\] 
\caption{}
    \label{tab:tab4}
\end{table}

\begin{table}[ht!]
    \centering
    \[
\begin{tblr}{|c|c|c|c|c|c|}
\hline
S' & h_3'(S')& \rho_3'(S')&S'&h_3'(S')&\rho_3'(S')\\
\hline
    \embeddpdf{-3.8}{3.8}{LR3aaa.pdf} & 0 &\embeddpdf{-3.8}{3.8}{LR3aaa.pdf} &\embeddpdf{-3.8}{3.8}{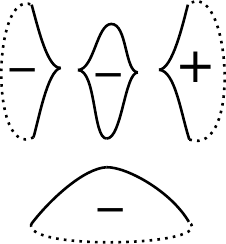} & \embeddpdf{-3.8}{3.8}{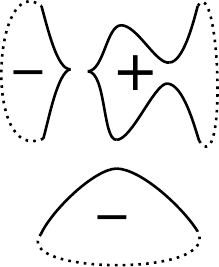} & -\embeddpdf{-3.8}{3.8}{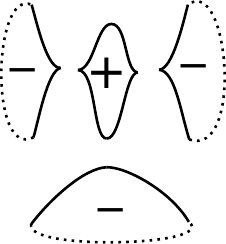}-\embeddpdf{-3.8}{3.8}{LR3bba--}\\
    \hline
    \embeddpdf{-3.8}{3.8}{LR3baa.pdf} & 0 & 0&\embeddpdf{-3.8}{3.8}{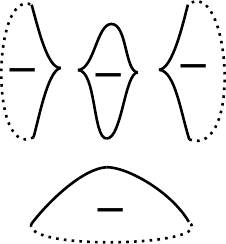} &  \embeddpdf{-3.8}{3.8}{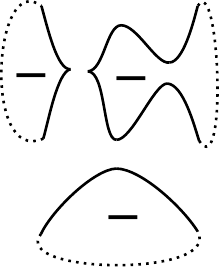}& 0 \\
    \hline
    \embeddpdf{-3.8}{3.8}{LR3aab.pdf} & 0 & \embeddpdf{-3.8}{3.8}{LR3aab.pdf}+\embeddpdf{-3.8}{3.8}{LR3baa.pdf} &\embeddpdf{-3.8}{3.8}{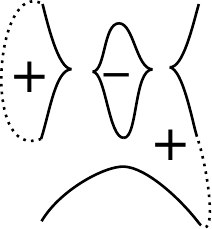} & \embeddpdf{-3.8}{3.8}{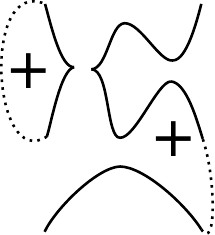}& -\embeddpdf{-3.8}{3.8}{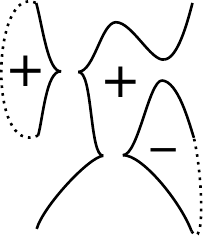}-\embeddpdf{-3.8}{3.8}{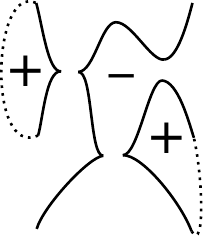}\\
    \hline
    \embeddpdf{-3.8}{3.8}{LR3aba.pdf} & 0 & \embeddpdf{-3.8}{3.8}{LR3aba.pdf}& \embeddpdf{-3.8}{3.8}{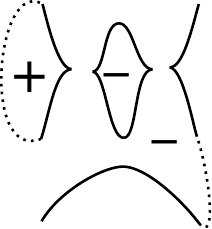} & \embeddpdf{-3.8}{3.8}{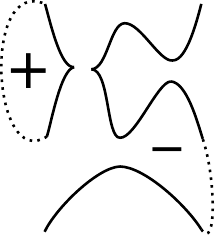} & -\embeddpdf{-3.8}{3.8}{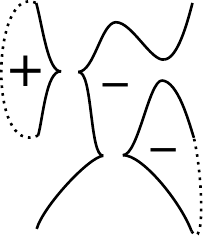}  \\
    \hline
    \embeddpdf{-3.8}{3.8}{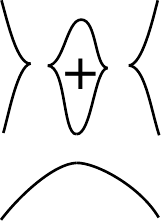} & 0 & 0 &\embeddpdf{-3.8}{3.8}{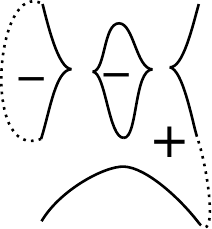} & \embeddpdf{-3.8}{3.8}{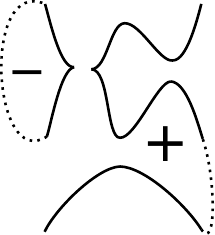} & -\embeddpdf{-3.8}{3.8}{LR3bba++-.pdf}-\embeddpdf{-3.8}{3.8}{LR3bba+-+.pdf}\\
    \hline
    \embeddpdf{-3.8}{3.8}{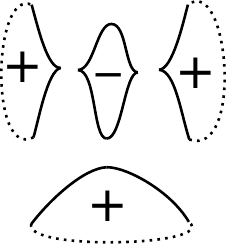} & \embeddpdf{-3.8}{3.8}{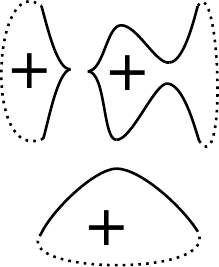} & -\embeddpdf{-3.8}{3.8}{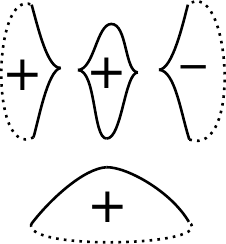}-\embeddpdf{-3.8}{3.8}{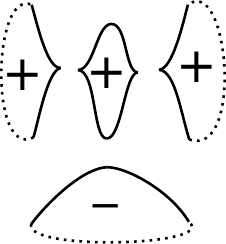}-\embeddpdf{-3.8}{3.8}{LR3bba++} &\embeddpdf{-3.8}{3.8}{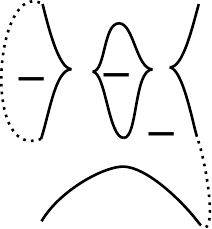} & \embeddpdf{-3.8}{3.8}{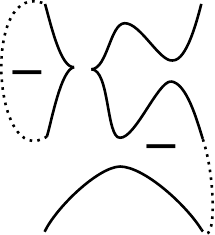} & -\embeddpdf{-3.8}{3.8}{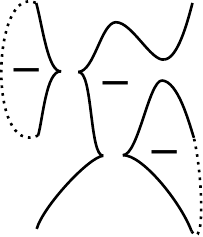}\\
    \hline
    \embeddpdf{-3.8}{3.8}{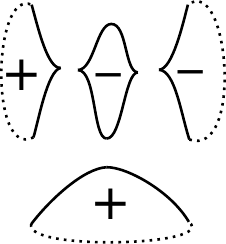} & \embeddpdf{-3.8}{3.8}{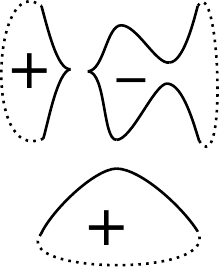} & -\embeddpdf{-3.8}{3.8}{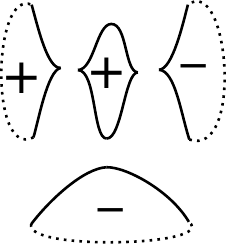}-\embeddpdf{-3.8}{3.8}{LR3bba+-} &\embeddpdf{-3.8}{3.8}{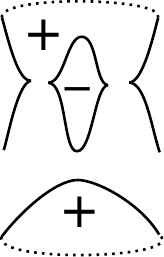} & \embeddpdf{-3.8}{3.8}{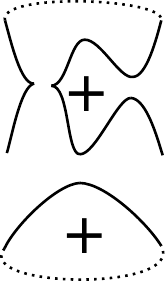} & -\embeddpdf{-3.8}{3.8}{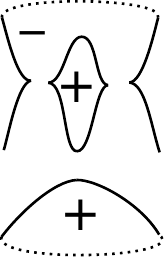}-\embeddpdf{-3.8}{3.8}{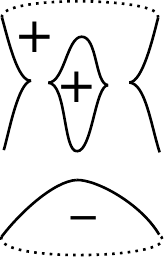}-\embeddpdf{-3.8}{3.8}{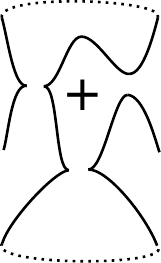} \\
    \hline
    \embeddpdf{-3.8}{3.8}{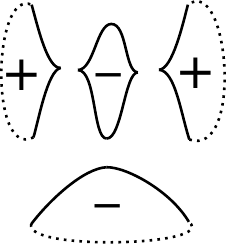} & \embeddpdf{-3.8}{3.8}{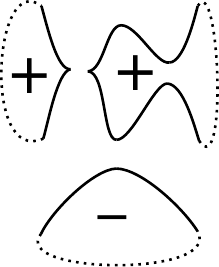}& -\embeddpdf{-3.8}{3.8}{LR3bab++--.pdf}-\embeddpdf{-3.8}{3.8}{LR3bba+-}&\embeddpdf{-3.8}{3.8}{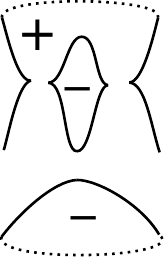} & \embeddpdf{-3.8}{3.8}{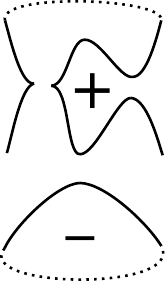} & -\embeddpdf{-3.8}{3.8}{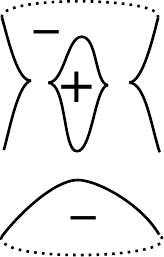}-\embeddpdf{-3.8}{3.8}{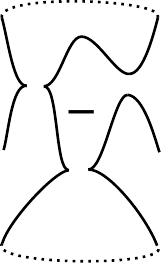} \\
    \hline
    \embeddpdf{-3.8}{3.8}{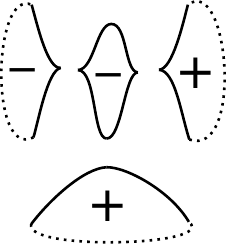} &  \embeddpdf{-3.8}{3.8}{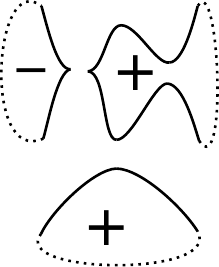}& -\embeddpdf{-3.8}{3.8}{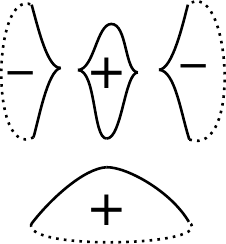}-\embeddpdf{-3.8}{3.8}{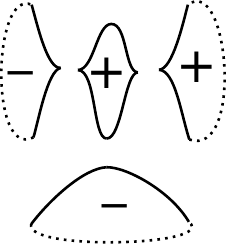}-\embeddpdf{-3.8}{3.8}{LR3bba-+}&\embeddpdf{-3.8}{3.8}{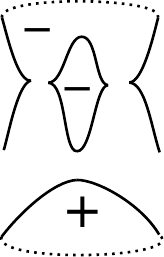} & \embeddpdf{-3.8}{3.8}{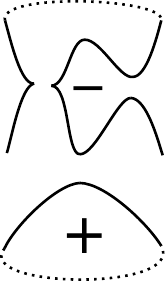} & -\embeddpdf{-3.8}{3.8}{LR3bab+--2.pdf}-\embeddpdf{-3.8}{3.8}{LR3bba-1.pdf} \\
    \hline
\end{tblr}
\]
\caption{}
    \label{tab:tab5}
\end{table}

\begin{table}[ht!]
    \centering
   \[
\begin{tblr}{|c|c|c|c|c|c|}
\hline
     S' & h_3'(S')& \rho_3'(S')&S' & h_3'(S') & \rho_3'(S')\\
      \hline
     \embeddpdf{-3.8}{3.8}{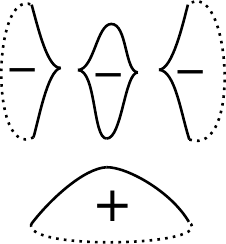} & \embeddpdf{-3.8}{3.8}{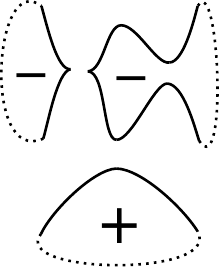} & -\embeddpdf{-3.8}{3.8}{LR3bab+---.pdf}-\embeddpdf{-3.8}{3.8}{LR3bba--} & \embeddpdf{-3.8}{3.8}{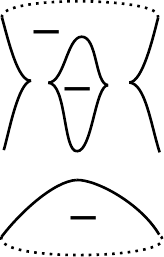} & \embeddpdf{-3.8}{3.8}{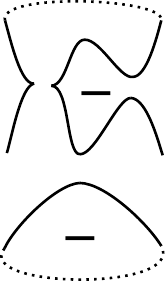} & 0\\
    \hline
     \embeddpdf{-3.8}{3.8}{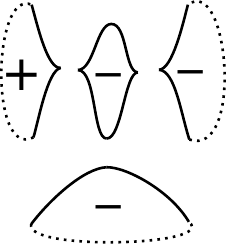} & \embeddpdf{-3.8}{3.8}{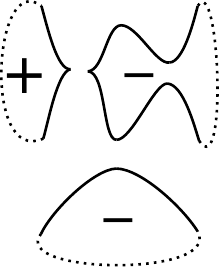} & 0&\embeddpdf{-3.8}{3.8}{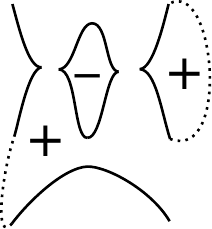} & \embeddpdf{-3.8}{3.8}{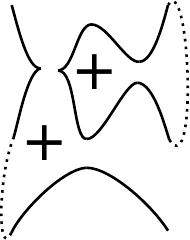} & -\embeddpdf{-3.8}{3.8}{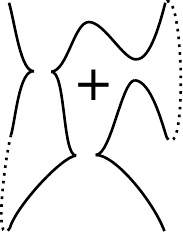}-\embeddpdf{-3.8}{3.8}{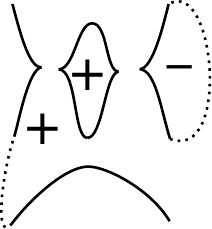}   \\
\hline
    \embeddpdf{-3.8}{3.8}{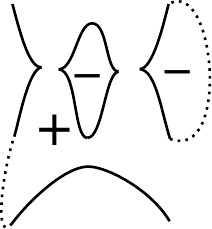} & \embeddpdf{-3.8}{3.8}{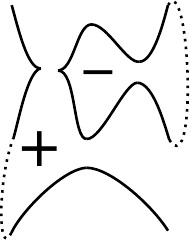} & -\embeddpdf{-3.8}{3.8}{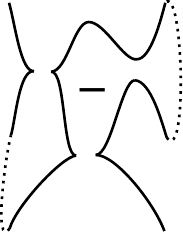}-\embeddpdf{-3.8}{3.8}{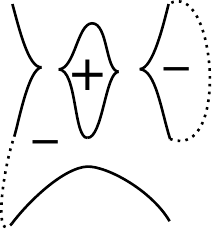} &\embeddpdf{-3.8}{3.8}{LR3bba+1.pdf} & 0 & \embeddpdf{-3.8}{3.8}{LR3bba+1.pdf}+\embeddpdf{-3.8}{3.8}{LR3bab++-2.pdf}+\embeddpdf{-3.8}{3.8}{LR3bab+-+2.pdf}\\
    \hline
    \embeddpdf{-3.8}{3.8}{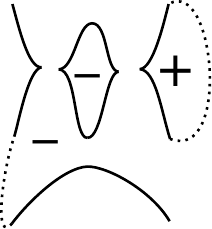} & \embeddpdf{-3.8}{3.8}{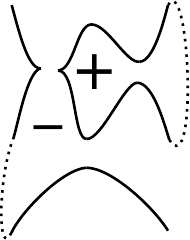} & -\embeddpdf{-3.8}{3.8}{LR3bba-.pdf} -\embeddpdf{-3.8}{3.8}{LR3bab+--.pdf}&\embeddpdf{-3.8}{3.8}{LR3bba-1.pdf} & 0 & \embeddpdf{-3.8}{3.8}{LR3bba-1.pdf}+\embeddpdf{-3.8}{3.8}{LR3bab+--2.pdf} \\
    \hline
    \embeddpdf{-3.8}{3.8}{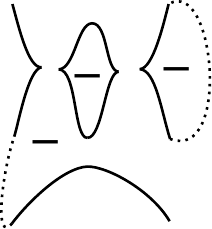} & \embeddpdf{-3.8}{3.8}{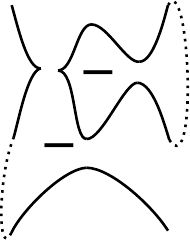} & 0 &\embeddpdf{-3.8}{3.8}{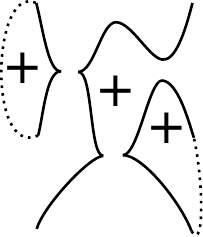} & 0 & \embeddpdf{-3.8}{3.8}{LR3bba+++.pdf}+\embeddpdf{-3.8}{3.8}{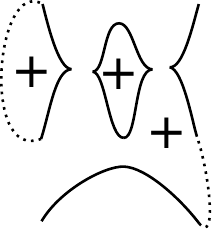}\\
    \hline
    \embeddpdf{-3.8}{3.8}{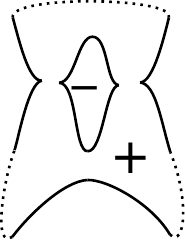} & \embeddpdf{-3.8}{3.8}{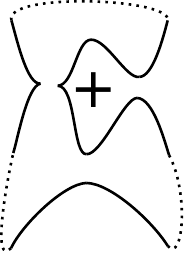} & -\embeddpdf{-3.8}{3.8}{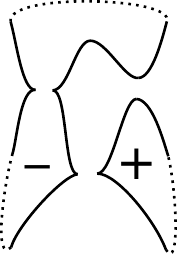}-\embeddpdf{-3.8}{3.8}{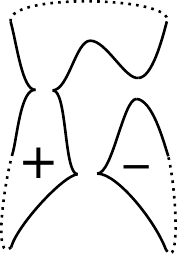}&\embeddpdf{-3.8}{3.8}{LR3bba+-+.pdf} & 0 & \embeddpdf{-3.8}{3.8}{LR3bba+-+.pdf}+\embeddpdf{-3.8}{3.8}{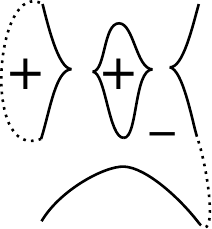} \\
    \hline
     \embeddpdf{-3.8}{3.8}{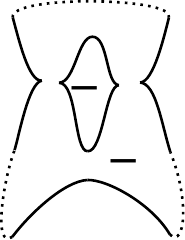} & \embeddpdf{-3.8}{3.8}{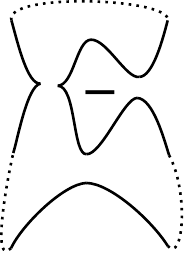} & -\embeddpdf{-3.8}{3.8}{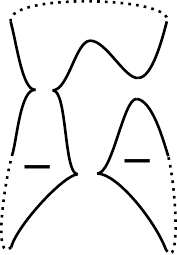}&\embeddpdf{-3.8}{3.8}{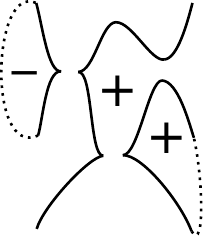} &0  & \embeddpdf{-3.8}{3.8}{LR3bba-++.pdf}+\embeddpdf{-3.8}{3.8}{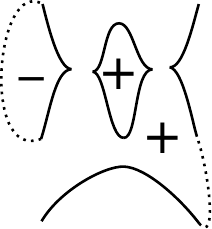}  \\
    \hline
     \embeddpdf{-3.8}{3.8}{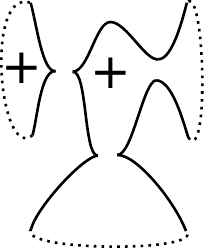} & 0 & \embeddpdf{-3.8}{3.8}{LR3bba++.pdf}+\embeddpdf{-3.8}{3.8}{LR3bab++-+.pdf} +\embeddpdf{-3.8}{3.8}{LR3bab+++-.pdf} &\embeddpdf{-3.8}{3.8}{LR3bba++-.pdf} & 0 & \embeddpdf{-3.8}{3.8}{LR3bba++-.pdf}+\embeddpdf{-3.8}{3.8}{LR3bab++-1.pdf}\\
     \hline
     \embeddpdf{-3.8}{3.8}{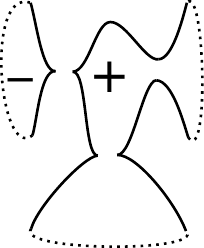} & 0 & \embeddpdf{-3.8}{3.8}{LR3bba-+.pdf} +\embeddpdf{-3.8}{3.8}{LR3bab+--+.pdf}+\embeddpdf{-3.8}{3.8}{LR3bab+-+-.pdf}&\embeddpdf{-3.8}{3.8}{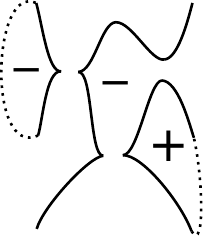} &  0& \embeddpdf{-3.8}{3.8}{LR3bba--+.pdf}+\embeddpdf{-3.8}{3.8}{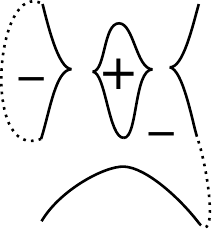} \\
    \hline
     \embeddpdf{-3.8}{3.8}{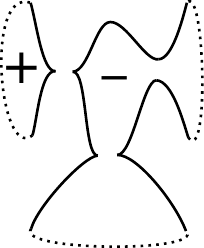} & 0 & \embeddpdf{-3.8}{3.8}{LR3bba+-.pdf}+\embeddpdf{-3.8}{3.8}{LR3bab++--.pdf}&\embeddpdf{-3.8}{3.8}{LR3bba+--.pdf} & 0 & \embeddpdf{-3.8}{3.8}{LR3bba+--.pdf} \\
     \hline
     \embeddpdf{-3.8}{3.8}{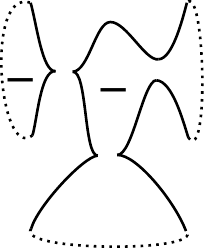} &0  & \embeddpdf{-3.8}{3.8}{LR3bba--.pdf}+\embeddpdf{-3.8}{3.8}{LR3bab+---.pdf} &\embeddpdf{-3.8}{3.8}{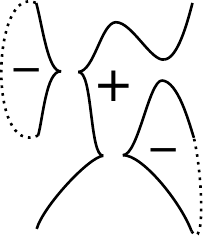} & 0 & \embeddpdf{-3.8}{3.8}{LR3bba-+-.pdf}+\embeddpdf{-3.8}{3.8}{LR3bab+--1.pdf}\\
    \hline
     \embeddpdf{-3.8}{3.8}{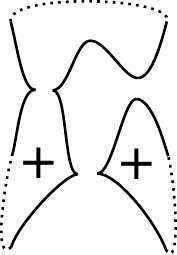} & 0 & \embeddpdf{-3.8}{3.8}{LR3bba++1.pdf}+\embeddpdf{-3.8}{3.8}{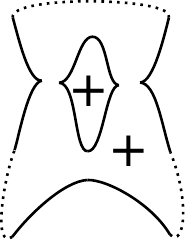}&\embeddpdf{-3.8}{3.8}{LR3bba---.pdf} & 0 & \embeddpdf{-3.8}{3.8}{LR3bba---.pdf}\\
     \hline
      \embeddpdf{-3.8}{3.8}{LR3bba+-1.pdf} &0  & \embeddpdf{-3.8}{3.8}{LR3bba+-1.pdf}+\embeddpdf{-3.8}{3.8}{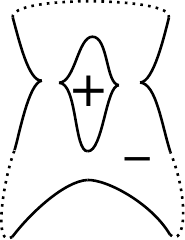}& \embeddpdf{-3.8}{3.8}{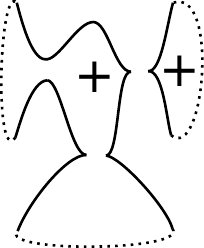} & 0 & \embeddpdf{-3.8}{3.8}{LR3abb++.pdf}+\embeddpdf{-3.8}{3.8}{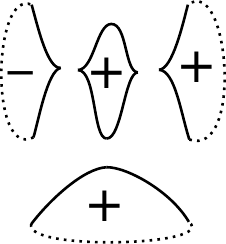}+\embeddpdf{-3.8}{3.8}{LR3bab++-+.pdf} \\
    \hline
     \embeddpdf{-3.8}{3.8}{LR3bba-+1.pdf} & 0 & \embeddpdf{-3.8}{3.8}{LR3bba-+1.pdf}+\embeddpdf{-3.8}{3.8}{LR3bab+-.pdf} &\embeddpdf{-3.8}{3.8}{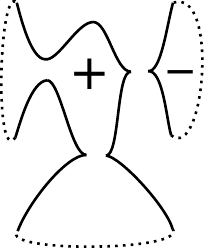} & 0 & \embeddpdf{-3.8}{3.8}{LR3abb+-.pdf}+\embeddpdf{-3.8}{3.8}{LR3bab++--.pdf}+\embeddpdf{-3.8}{3.8}{LR3bab+-+-.pdf}\\
     \hline
     \embeddpdf{-3.8}{3.8}{LR3bba--1.pdf} &0  & \embeddpdf{-3.8}{3.8}{LR3bba--1.pdf} &\embeddpdf{-3.8}{3.8}{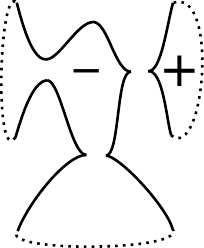} &0  & \embeddpdf{-3.8}{3.8}{LR3abb-+.pdf}+\embeddpdf{-3.8}{3.8}{LR3bab+--+.pdf}\\
    \hline
     \embeddpdf{-3.8}{3.8}{LR3bba+.pdf} & 0 & \embeddpdf{-3.8}{3.8}{LR3bba+.pdf}+\embeddpdf{-3.8}{3.8}{LR3bab++-.pdf}+\embeddpdf{-3.8}{3.8}{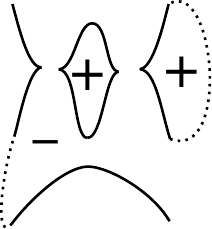}& \embeddpdf{-3.8}{3.8}{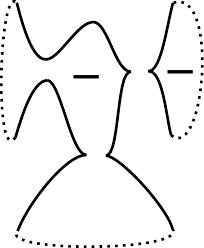} & 0 & \embeddpdf{-3.8}{3.8}{LR3abb--.pdf} +\embeddpdf{-3.8}{3.8}{LR3bab+---.pdf} \\
     \hline
\end{tblr}
\]
    \caption{}
    \label{tab:tab6}
\end{table}

\begin{table}[ht!]
    \centering
   \[
\begin{tblr}{|c|c|c|c|c|c|}
\hline
     S' & h_3'(S')& \rho_3'(S')&S' & h_3'(S') & \rho_3'(S')\\
\hline
 \embeddpdf{-3.8}{3.8}{LR3bba-.pdf} & 0 & \embeddpdf{-3.8}{3.8}{LR3bba-.pdf}+\embeddpdf{-3.8}{3.8}{LR3bab+--.pdf} &\embeddpdf{-3.8}{3.8}{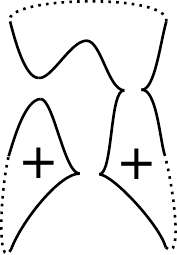} & 0 & \embeddpdf{-3.8}{3.8}{LR3abb++1.pdf}+\embeddpdf{-3.8}{3.8}{LR3bab++.pdf}\\
     \hline
     \embeddpdf{-3.8}{3.8}{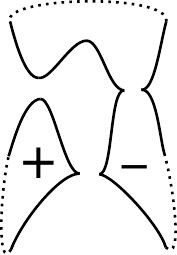} &0  & \embeddpdf{-3.8}{3.8}{LR3abb+-1.pdf}+\embeddpdf{-3.8}{3.8}{LR3bab+-.pdf}&\embeddpdf{-3.8}{3.8}{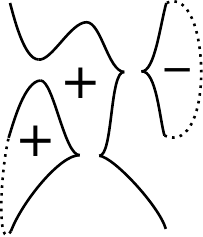} & 0 & \embeddpdf{-3.8}{3.8}{LR3abb++-.pdf}+\embeddpdf{-3.8}{3.8}{LR3bab++-.pdf} \\
    \hline
     \embeddpdf{-3.8}{3.8}{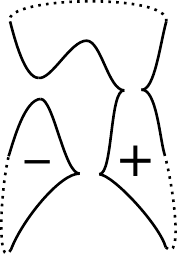} &  0& \embeddpdf{-3.8}{3.8}{LR3abb-+1.pdf}+\embeddpdf{-3.8}{3.8}{LR3bab+-.pdf}&\embeddpdf{-3.8}{3.8}{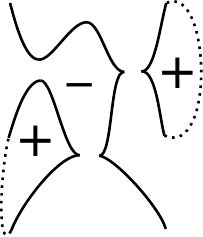} & 0 & \embeddpdf{-3.8}{3.8}{LR3abb+-+.pdf}+\embeddpdf{-3.8}{3.8}{LR3bab+-+.pdf} \\
     \hline
     \embeddpdf{-3.8}{3.8}{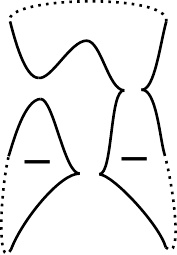} & 0 & \embeddpdf{-3.8}{3.8}{LR3abb--1.pdf} &\embeddpdf{-3.8}{3.8}{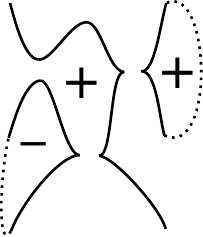} & 0 & \embeddpdf{-3.8}{3.8}{LR3abb-++.pdf}+\embeddpdf{-3.8}{3.8}{LR3bab+-+.pdf} \\
    \hline
     \embeddpdf{-3.8}{3.8}{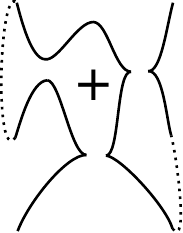} & 0 & \embeddpdf{-3.8}{3.8}{LR3abb+.pdf}+\embeddpdf{-3.8}{3.8}{LR3bab++-1.pdf}+\embeddpdf{-3.8}{3.8}{LR3bab+-+1.pdf}&\embeddpdf{-3.8}{3.8}{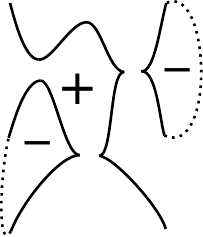} &  0& \embeddpdf{-3.8}{3.8}{LR3abb-+-.pdf}+\embeddpdf{-3.8}{3.8}{LR3bab+--.pdf} \\
     \hline
     \embeddpdf{-3.8}{3.8}{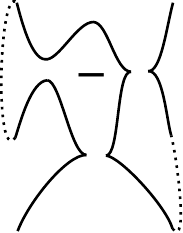} & 0 & \embeddpdf{-3.8}{3.8}{LR3abb-.pdf}+\embeddpdf{-3.8}{3.8}{LR3bab+--1.pdf}&\embeddpdf{-3.8}{3.8}{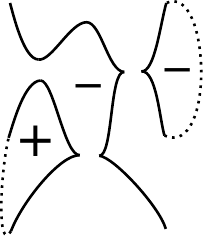} & 0 & \embeddpdf{-3.8}{3.8}{LR3abb+--.pdf}+\embeddpdf{-3.8}{3.8}{LR3bab+--.pdf} \\
     \hline
     \embeddpdf{-3.8}{3.8}{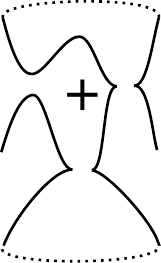} & 0 & \embeddpdf{-3.8}{3.8}{LR3abb+1.pdf}+\embeddpdf{-3.8}{3.8}{LR3bab++-2.pdf}+\embeddpdf{-3.8}{3.8}{LR3bab+-+2.pdf} &\embeddpdf{-3.8}{3.8}{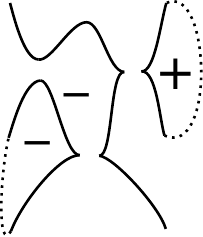} &0  & \embeddpdf{-3.8}{3.8}{LR3abb--+.pdf}\\
     \hline
      \embeddpdf{-3.8}{3.8}{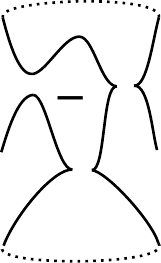} & 0 & \embeddpdf{-3.8}{3.8}{LR3abb-1.pdf}+\embeddpdf{-3.8}{3.8}{LR3bab+--2.pdf}& \embeddpdf{-3.8}{3.8}{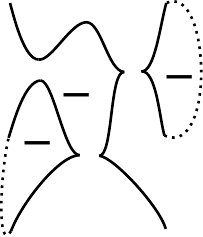} & 0 & \embeddpdf{-3.8}{3.8}{LR3abb---.pdf}  \\
     \hline
     \embeddpdf{-3.8}{3.8}{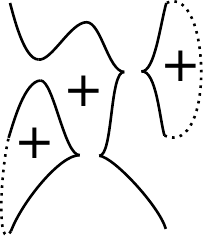} & 0 & \embeddpdf{-3.8}{3.8}{LR3abb+++.pdf}+\embeddpdf{-3.8}{3.8}{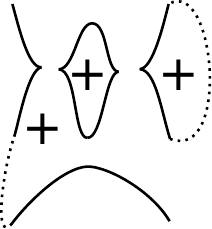}&  \embeddpdf{-3.8}{3.8}{LR3bbb.pdf} &  \embeddpdf{-3.8}{3.8}{LR3bab+.pdf}& 0  \\
     \hline
\end{tblr}
\]
    \caption{}
    \label{tab:tab7}
\end{table}

\clearpage
\bibliographystyle{plain}
\bibliography{refrences}

\end{document}